\documentclass[12pt,oneside,reqno]{amsart}

\usepackage{amsmath,amssymb,amsthm,amsfonts,latexsym, amsthm, amscd, mathrsfs, stmaryrd}
\usepackage[colorlinks=true,linkcolor=blue,citecolor=red]{hyperref}
\usepackage{graphicx,color,colortbl}
\usepackage{upgreek}
\usepackage[mathscr]{euscript}
\usepackage{hhline}
\numberwithin{equation}{section}
\usepackage{dynkin-diagrams}
\usepackage{epic}
\allowdisplaybreaks
\numberwithin{equation}{section}

\usepackage{tikz}
\usetikzlibrary{
    cd,
	shapes,
	arrows,
	positioning,
	decorations.markings,
	decorations.pathmorphing,
	circuits.logic.US,
	circuits.logic.IEC,
	fit,
	calc,
	plotmarks,
	matrix
}

\tikzset{
>=stealth',
help lines/.style={dashed, thick},
axis/.style={<->},
important line/.style={thick},
connection/.style={thick, dotted},
punkt/.style={
rectan\mathrm{GL}e,
rounded corners,
draw=black, thick,
text width=4.5em,
minimum height=2em,
text centered,
},
pil/.style={
->,
thick,
gray,
shorten <=2pt,
shorten >=2pt,}
}

\usepackage{array}
\usepackage{adjustbox}
\usepackage{cleveref}
\usepackage{enumerate}

\usepackage[DIV=11]{typearea}
	\newtheorem{theorem}{Theorem}[section]

	\newtheorem{corollary}[theorem]{Corollary}

	\newtheorem{lemma}[theorem]{Lemma}

	\newtheorem{proposition}[theorem]{Proposition}

	\numberwithin{equation}{section}
	
	\theoremstyle{remark}
	\newtheorem{remark}[theorem]{Remark}
\newtheorem{innercustomthm}{{\bf Theorem}}
\newenvironment{customthm}[1]
  {\renewcommand\theinnercustomthm{#1}\innercustomthm}
  {\endinnercustomthm}

\makeatletter

\newcommand{\arxiv}[1]{\href{http://arxiv.org/abs/#1}{\tt arXiv:\nolinkurl{#1}}}

\newcommand{\Rmnum}[1]{\expandafter\@slowromancap\romannumeral #1@}

\def \gl{\mathfrak{gl}}
\def \sl{\mathfrak{sl}}
\def \so{\mathfrak{so}}
\def \sp{\mathfrak{sp}}

\def \bC{\mathbb{C}}
\def \N{\mathbb{N}}
\def \Q{\mathbb{Q}}
\def \Z{\mathbb{Z}}
\def \I{\mathbb{I}}

\def \cS{{\mathcal S}}

\def \bs{\mathbf{r}}
\def \tk{\widetilde{k}}
\def \fX{ \Upsilon}

\def \Mat{\mathrm{Mat}}

\newcommand{\U}{\mathbf{U}}
\newcommand{\Ui}{(\U)^\imath}

\newcommand{\tUi}{\widetilde{{\mathbf U}}^\imath}

\newcommand{\qbinom}[2]{\begin{bmatrix} #1\\#2 \end{bmatrix} }

\def \ov{\overline}

\def \ba{\mathbf{a}}
\def \bb{\mathbf{b}}

\def \hf{\frac{1}{2}}
\def \la{\lambda}
\newcommand{\nc}{\newcommand}
\nc{\browntext}[1]{\textcolor{brown}{#1}}
\nc{\greentext}[1]{\textcolor{green}{#1}}
\nc{\redtext}[1]{\textcolor{red}{#1}}
\nc{\bluetext}[1]{\textcolor{blue}{#1}}
\nc{\brown}[1]{\browntext{ #1}}
\nc{\green}[1]{\greentext{ #1}}
\nc{\red}[1]{\redtext{ #1}}
\nc{\blue}[1]{\bluetext{ #1}}

\def \A{\mathbb{A}}
\def \Q {\mathbb Q}
\def \TT{\mathbf T}
\def \tTT{\widetilde{\mathbf T}}
\def \uTT{\underline{\mathbf T}}

\def \tz{\widetilde{z}}
\def \cT{\mathcal{T}}

\def \U{\mathbf U}
\def \Ui{\mathbf{U}^\imath}
\def \Uj{\mathbf{U}^\jmath}

\def \cH{\mathcal{H}}
\def \cK{\mathcal{K}}
\def \End{\mathrm{End}}
\def \tOmega{\widetilde{\Omega}}
\def \cV{\mathcal{V}}
\def \Par{\mathrm{Par}}
\def \jPar{\mathrm{Par}^\jmath}
\def \iPar{\mathrm{Par}^\imath}
\def \jPhi{\Lambda^\jmath}
\def \iPhi{\Lambda^\imath}

\def \fkc{\mathfrak{c}}
\def \tg{\widetilde{g}}

\def \RUi{\mathrm{U}^\imath}
\def \RUj{\mathrm{U}^\jmath}

\def \un{\underline}

\newcommand{\dv}[2]{{B}_{#1}^{{(#2)}}}
\newcommand{\edvi}[1]{{B}_{0,\bar{0}}^{(#1)}}
\newcommand{\odvi}[1]{{B}_{0,\bar{1}}^{(#1)}}

\author[Weinan Zhang]{Weinan Zhang}
\address{Department of Mathematics and New Cornerstone Science Laboratory, The University of Hong Kong, Pokfulam, Hong Kong SAR, P.R.China}
\email{mathzwn@hku.hk}



\begin{document}

\title{Quantum Howe duality and Schur duality of type AIII}

\maketitle

\begin{abstract} 
We establish a new connection between the iHowe duality of type AIII established by Luo-Xu and the iSchur duality established by Bao-Wang. We show that iweight $\ov{\rho}$ space in the iHowe duality is naturally isomorphic to the tensor space in the iSchur duality. Under this isomorphism, we show that the relative braid group action on this iweight space coincides with the action of the type B Hecke algebra in the iSchur duality. As a consequence, we derive from multiplicity-free decompositions that the iweight $\ov{\rho}$ spaces of irreducible modules over iquantum groups are irreducible modules over the type B Hecke algebra. Meanwhile, in the iHowe duality, we identify the relative braid group action from one side with the action of $K$-matrices and $R$-matrices from the other side.
\end{abstract}

\setcounter{tocdepth}{1}
\tableofcontents

\section{Introduction}
\subsection{Background}
In the classical invariant theory, there are two distinct types of dualities: the Schur duality and the Howe duality. 
It is well-known that the Schur duality and the (symmetric) Howe duality for general linear groups imply each other; for instance, see \cite{Ho92, CW12}. %involves multiplicity-free actions for pairs of reductive groups. 

The quantum analog of the Schur duality was established by Jimbo \cite{Jim86} and referred as the Schur-Jimbo duality. He established a double centralizer property between the type A Hecke algebra and the quantum general linear group $\U(\gl_M)$. The iquantum groups are coideal subalgebras of quantum groups, arising from the theory of quantum symmetric pairs introduced by Letzter \cite{Let02}. The iSchur duality, constructed by Bao-Wang \cite{BW18a}, asserts the double centralizer property between the type B/C Hecke algebras and the iquantum groups of (quasi-split) type AIII. This construction was generalized to type D Hecke algebra in \cite{Bao17} and to non-quasi-split type in \cite{SW23}. More recently, dualities between the $q$-Brauer algebra and the iquantum groups of type AI/AII were obtained in \cite{CS24}.

On the other hand, various formulations of the quantum Howe dualities have appeared in the literature. For type A, the multiplicity-free action for a pair of quantum general linear groups was constructed by Zhang \cite{Zha02}. He further showed that the quantum Howe duality implies the Schur-Jimbo duality. In a different approach, Toledano Laredo \cite{TL02} established a quantum version of the dual pair $(\gl_M,\gl_N)$ and studied the braid group actions in this duality. Using the web categories, Cautis-Kamnitzer-Morrison \cite{CKM14} established a type A skew $q$-Howe duality.

Beyond type A, Letzter's iquantum groups naturally arise in the $q$-Howe dualities. 
Noumi-Umeda-Wakayama \cite{NUW96} constructed quantum analogues of the
dual pairs ($\sl_2$, $\so_n$) and ($\sp_2$, $\so_n$), where the non-standard quantum groups $U_q'(\so_n)$ appears. It is now understood that the non-standard quantum groups are special cases of iquantum groups. Using the web categories, Sartori-Tubbenhauer \cite{ST19} established the $q$-Howe dualities of type BCD, where usual quantum groups appear on one side and the iquantum groups appear on the other side. A quantum super version of type BCD How dualities was recently established in \cite{BaKw25}.

Via a geometric approach, Wang realized in \cite{Wa01} the Howe duality of type A using the type A partial flag varieties. In \cite{LX22}, Luo-Xu generalized this construction to type BCD partial flag varieties and obtained a duality for a pair of iquantum groups of quasi-split type AIII, which we refer to as the iHowe duality.

\subsection{Goal and Scope} The goal is to show that one can algebraically recover the iSchur duality in \cite{BW18a} from the iHowe duality of type AIII in \cite{LX22}. This is a generalization of the result in \cite{Zha02} that the Schur-Jimbo duality can be recovered from the $q$-Howe duality of type A. The key ingredient is the relative braid group symmetries established in \cite{WZ23,WZ25}.%Meanwhile, we show that the $R$-matrices and $K$-matrices induce a braid group action in the iHowe duality, and this action can be identified with the relative braid group action from the other side.

In this paper, we will consider iquantum groups of quasi-split type AIII. Following the convention in \cite{BW18a}, one can separate them into two classes $\Ui_n$ and $\Uj_n$, which are associated to symmetric pairs $(\gl_{2n},\gl_n\oplus\gl_n)$ and $(\gl_{2n+1},\gl_n\oplus\gl_{n+1})$ respectively; see also Section~\ref{sec:iQG}. Due to this phenomenon, there are four versions of type AIII iHowe dual pairs obtained in \cite{LX22}
$$
(\Ui_m,\Ui_n), (\Uj_m,\Uj_n),(\Ui_m,\Uj_n), (\Uj_m,\Ui_n).
$$
We will consider the first two pairs in this paper and the other two pairs can be treated in a similar way. 

Although details and formulas for treating $(\Ui_m,\Ui_n)$ and $(\Uj_m,\Uj_n)$ are quite different from each other, the results turn out to be parallel in these two cases. Thus, in the introduction, we will use the uniform notation $\U^{\fkc}_n$ where $\fkc=\imath$ or $\jmath$ and explain these two cases together.

% to denote either $\Ui_n$ or $\Uj_n$
%\subsection{Main results}
\subsection{The iweight $\ov{\rho}$ space}
Let $\fkc=\imath$ or $\jmath$. Recall from \cite{LX22} the following iHowe duality of type AIII (see also Sections~\ref{sec:Uiduality} and \ref{sec:Ujduality} for details)
\begin{align}\label{intro:Howe}
\U^{\fkc}_m \curvearrowright \mathcal{V}_{m|n,d}^\fkc \curvearrowleft \U^{\fkc}_n
\end{align}
Let $V$ be the natural representation of $\U(\gl_M)$ and view it as a $\U^\fkc_m$-module for $m=\lfloor\frac{M}{2}\rfloor$ by restriction. Let $\cH_B(n)$ be the type B$_n$ Hecke algebra. Recall from \cite{BW18a} the iSchur duality
\begin{align}\label{intro:Schur}
\U^{\fkc}_m \curvearrowright V^{\otimes n} \curvearrowleft \cH_B(n)
\end{align}

In order to establish a connection between \eqref{intro:Howe} and \eqref{intro:Schur}, we first set $d=n$ and introduce the $\U^{\fkc}_n$-weight $\ov{\rho}$ in \eqref{def:rho}. This $\U^{\fkc}_n$-weight $\ov{\rho}$ plays a similar role as the weight $0$ in the theory of quantum groups; in particular, one can check that $\ov{\rho}$ is invariant under the action of the relative Weyl group; see Section~\ref{sec:braid}.

Then we consider the $\U^{\fkc}_n$-weight space $\cV_{m|n,n}^{\fkc,0}$ associated to $\ov{\rho}$ for the $\U^{\fkc}_n$-module $\cV_{m|n,n}^{\fkc}$. Due to commuting actions in \eqref{intro:Howe}, the $\U^{\fkc}_n$-weight space $\cV_{m|n,n}^{\fkc,0}$ admits an action of $\U^{\fkc}_m$. The first result tells that this $\U^{\fkc}_m$-module can be identified with $V^{\otimes n}$. 

 \begin{customthm} {\bf A}
  [Theorem~\ref{thm:0wti}, Theorem~\ref{thm:0wtj}]
  \label{thm:A}
 Let $\fkc=\imath$ or $\jmath$. There is an explicit isomorphism of $\U^{\fkc}_m$-modules $\cV_{m|n,n}^{\fkc,0}\cong V^{\otimes n}$.
\end{customthm}

\subsection{Relative braid group symmetries}

The second step is to build a Hecke algebra $\cH_B(n)$-action on the $\U^{\fkc}_n$-weight space $\cV_{m|n,n}^{\fkc,0}$. Recall that, in the type A quantum Howe duality, Zhang \cite{Zha02} showed that Lusztig braid group symmetries (cf. \cite{Lus94}) induce a (type A) Hecke algebra action on the $0$-weight space, and this recovers the Hecke algebra action in the Schur-Jimbo duality. 

As a generalization of Lusztig symmetries, the relative braid group symmetries on iquantum groups were uniformly constructed for arbitrary finite type by Wang and the author \cite{WZ23}; see also \cite{KP11,Do20} for some previous case-by-case constructions and \cite{LW22} for a Hall algebra construction. Recently, in \cite{WZ25}, relative braid group symmetries on integrable modules over quasi-split iquantum groups were constructed via closed formulas, and these symmetries are shown to be integral and compatible with the symmetries on iquantum groups.

Let $\TT_i,0\le i \le n-1$ denote the relative braid group symmetries on the integrable $\U^\fkc_n$-module $\cV_{m|n,n}^{\fkc}$; see \cite{WZ25} and Proposition~\ref{prop:braidmod}. Since $\ov{\rho}$ is fixed by the relative Weyl group action, $\TT_i$ preserve the weight space $\cV_{m|n,n}^{\fkc,0}$. Note that, for quasi-split type AIII, the relative Weyl group is isomorphic to the type B Weyl group, and hence $\TT_i$ satisfy type B braid relations.

 \begin{customthm} {\bf B}
  [Proposition~\ref{prop:TiHecke}, Theorem~\ref{thm:T0Hecke}, Theorem~\ref{thm:T0Heckej},  Proposition~\ref{prop:TiHeckej}]
  \label{thm:B}
Let $\fkc=\imath$ or $\jmath$. Relative braid group symmetries $\TT_i$ on $\cV_{m|n,n}^{\fkc,0}$ satisfy the Hecke relation, and hence they induce an action of the type B Hecke algebra $\cH_{B}(n)$ on $\cV_{m|n,n}^{\fkc,0}$. Moreover, under the isomorphism $\cV_{m|n,n}^{\fkc,0}\cong V^{\otimes n}$ in Theorem~\ref{thm:A}, this $\cH_{B}(n)$-action on $\cV_{m|n,n}^{\fkc,0}$ coincides with the $\cH_{B}(n)$-action in \eqref{intro:Schur} on $V^{\otimes n}$.
\end{customthm}

Theorem~\ref{thm:B} shows that one can recover the iSchur duality \eqref{intro:Schur} from the iHowe duality \eqref{intro:Howe}. In addition, symmetries $\TT_i$ generate the centralizer algebra $\End_{\U^\fkc_m}\big(\cV_{m|n,n}^{\fkc,0}\big)$; see Corollaries~\ref{cor:Endi} and \ref{cor:Endj}.

\subsection{Multiplicity-free decompositions}

Let $\widetilde{L}_\la^{[n],\fkc}$ denote the irreducible $\U^{\fkc}_n$-module parametrized by pairs of partitions $\la\in \Par^{\fkc}_n(n)$; cf. \cite{Wat21}. The multiplicity-free decompositions of \eqref{intro:Howe}  in terms of irreducible bimodules were given in \cite[Section 6]{LX22}. 
As a consequence of Theorem~\ref{thm:A}-\ref{thm:B} and the uniqueness of these multiplicity-free decompositions, we have the following result.

 \begin{customthm} {\bf C}
  [Theorem~\ref{thm:mfdecomp}]
  \label{thm:C}
Let $\fkc=\imath$ or $\jmath$. The $\U^{\fkc}_n$-weight space $(\widetilde{L}_\la^{[n],\fkc})_{\ov{\rho}}$ for any $\lambda\in \Par^{\fkc}_n(n)$ is an irreducible $\cH_B(n)$-module. 
\end{customthm}

Recall from \cite{DJ92} that irreducible $\cH_B(n)$-modules are parameterized by pairs of partitions. When $\fkc=\jmath$, the multiplicity-free decomposition of \eqref{intro:Schur} in terms of irreducible $(\Uj_m,\cH_B(n))$-bimodules was given in \cite{Wat21}. In this case, we are able to identify $(\widetilde{L}_\la^{[n],\jmath})_{\ov{\rho}}$ with the irreducible $\cH_B(n)$-module associated to the same $\la$; see Section~\ref{sec:decompSchur} and Theorem~\ref{thm:mfdecomp}.

\subsection{Relative braid group actions and $K$-matrix}
Let $M=2m$. In Theorem~\ref{thm:Uniso}, we obtain the following commutative diagrams involving see-saw Howe dual pairs, where the first row is the $q$-Howe duality of type A in \cite{TL02,Zha02} and the second row is the iHowe duality of type AIII in \cite{LX22}
\begin{equation}
\begin{tikzcd}
  & \U(\gl_M)
  & \curvearrowright\qquad  \cV_{M|n} \qquad \curvearrowleft
  &  \arrow[hookrightarrow]{d}{\iota_n}\U(\sl_n) \\
  & \Ui_m \arrow[hookrightarrow]{u}
  & \curvearrowright\qquad  \cV_{m|n}^\imath \arrow{u}{\tOmega}\qquad \curvearrowleft
  &\Ui_n
\end{tikzcd}
\end{equation}

In \cite{TL02}, Toledano Laredo showed that the $R$-matrices associated to $\U(\gl_M)$ induce a (type A) braid group action on $\cV_{M|n}$. Up to scalars, he further identified these $R$-matrices with Lusztig's braid group symmetries associated to $\U(\sl_n)$ when acting on $\cV_{M|n}$.

The $K$-matrix first appeared in Sklyanin's study of integrable system with boundary conditions \cite{Sk88}. In the theory of quantum symmetric pairs, a variant of the universal $K$-matrix of type AIII (without the reflection equation) first appeared in the work of Bao-Wang \cite{BW18a} as an isomorphism of modules over iquantum groups; the general construction of such isomorphisms was established in \cite{BW18b} and this construction works for arbitrary Kac-Moody type. Balagovic-Kolb \cite{BK19} constructed the universal $K$-matrix and showed that they are solutions to the reflection equation for finite type. Appel-Vlaar \cite{AV20} reformulated Balagovic-Kolb's construction via cylindrical twists and extended it to arbitrary Kac-Moody type. %of universal $K$-matrices and the reflection equation

Following Toledano Laredo's approach, we construct in Section~\ref{sec:reflec} a type B braid group action on $\cV_{M|n}$ using $R$-matrices and $K$-matrices associated to $\big(\U(\gl_M),\Ui_m\big)$. The last result allows one to identify this action with the relative braid group action associated to $\Ui_n$.

 \begin{customthm} {\bf D}
  [Theorem~\ref{thm:braidK}, Corollary~\ref{cor:braidactions}]
  \label{thm:D} 
  Up to scalars, the braid group action constructed using $R$-matrices and $K$-matrices of $\big(\U(\gl_M),\Ui_m\big)$ coincides with the relative braid group action induced by $\TT_i,i\in[0,n-1]$.%the action of relative braid group symmetry $\TT_0$ on $\cV_{m|n}^\imath$ coincides with the action of the (universal) $K$-matrix. 
\end{customthm}

\subsection{Organization} This paper is organized as follows. In Section~\ref{sec:pre}, we recall basics for iquantum groups of type AIII, the relative braid group symmetries on their integrable modules, and the iSchur duality. Sections \ref{sec:iHowe}-\ref{sec:jHowe} are two parallel sections. In Section~\ref{sec:iHowe}, we study the iHowe dual pair $(\Ui,\Ui)$ and connect it with the iSchur duality. In Section~\ref{sec:jHowe}, we study the iHowe dual pair $(\Uj,\Uj)$ and connect it with the iSchur duality.

In Section~\ref{sec:decomp}, we recall multiplicity-free decompositions for iHowe dualities and the iSchur duality, and then derive from them that the iweight $\ov{\rho}$ spaces of irreducible modules over iquantum groups are irreducible module over the type B Hecke algebra. In Section~\ref{sec:K}, we construct a braid group action in the iHowe duality using $R$-matrices and $K$-matrices, and then identify this action with the relative braid group action.

\vspace{0.2in}

\noindent {\bf Acknowledgement: } The author thanks Weiqiang Wang for proposing this problem and many helpful discussions. The author thanks Li Luo and Zheming Xu for explaining their paper. He also thanks Chun-Ju Lai and an anonymous referee for valuable comments. The author is partially supported by the New Cornerstone Foundation through the New Cornerstone Investigator grant awarded to Xuhua He.

\section{Preliminaries}\label{sec:pre}

In this section, we recall the basics for quantum symmetric pairs of type AIII. We recall from \cite{WZ25} the relative braid group action on integrable modules over quantum symmetric pairs, as well as the iSchur duality in \cite{BW18a}.

\subsection{Quantum general linear group}\label{sec:QG}

Let $q$ be an indeterminate. We denote, for $m,r\in\N$,
\begin{align}
[r] =\frac{q^r-q^{-r}}{q-q^{-1}},
 \quad
[r]!=\prod_{s=1}^r [s], \quad 
\qbinom{m}{r} =\frac{[m]! }{[r]! [m-r]!}.
\end{align}
Let $\Q(q)$ be the field of rational functions in $q$.
For $A, B$ in an $\Q(q)$-algebra, we write $[A,B]_{q^a} =AB -q^aBA$, and $[A,B] =AB - BA$. Denote by $A^{(n)}$ for the divided power $\frac{A^n}{[n]!}$.

For $a,b\in \hf\Z$ such that $b-a\in \N$, we use $[a,b]$ to denote the set $\{a,a+1,\ldots,b\}$.
For $N\in \N$, we use the following index sets
\begin{align}
\I_{N}&= \Big[-\frac{N-1}{2},\frac{N-1}{2}\Big].
%\qquad \I_{2n}=[-n+\hf,n-\hf].%\{-n+\frac{1}{2},-n+\frac{3}{2},\ldots,-\frac{1}{2},\frac{1}{2},\ldots,n-\frac{3}{2},n-\frac{1}{2}\}.
\end{align}

The quantum general linear group $\U_N=\U(\gl_N)$ is defined to be the $\Q(q)$-algebra generated by $E_i,F_i$ for $i\in \I_{N-1}$, and $D_a^{\pm1}$, $a\in \I_{N}$, subject to the following relations:
\begin{align}
D_a^{\pm1} D_b^{\pm1} &= D_b^{\pm1} D_a^{\pm1},\qquad D_a D_a^{-1} =1,
\label{eq:DD}
\\
D_a E_i &=q^{\delta_{a,i-\frac{1}{2}}-\delta_{a,i+\frac{1}{2}}} E_i D_a, \qquad D_a F_i =q^{-\delta_{a,i-\frac{1}{2}}+\delta_{a,i+\frac{1}{2}}} F_i D_a,
\\
[E_i,F_j] &=\delta_{ij}\frac{K_i-K_i^{-1}}{q-q^{-1}}, \qquad \text{ where } K_i=D_{i-\frac{1}{2}} D_{i+\frac{1}{2}}^{-1},
\label{Q4}
\end{align}
and the quantum Serre relations, for $i\neq j$,
\begin{align}
 E_i E_j = E_j E_i,& \qquad F_i F_j = F_j F_i, \qquad \text{ if } |i-j|>1.
\\
  E_i^{(2)} E_j +E_j  E_i^{(2)}  = E_i E_j E_i,&
\qquad
 F_i^{(2)} F_j +F_j  F_i^{(2)}  = F_i F_j F_i, \qquad \text{ if } |i-j|=1.
  \label{eq:serre}
\end{align} 
The quantum general linear group $\U_N$ admits a Hopf algebra structure with the comultiplication $\Delta$ given by
\[
\Delta(E_i)=E_i\otimes 1+K_i\otimes E_i,\quad
\Delta(F_i)=F_i\otimes K_i^{-1}+1\otimes F_i,\quad
\Delta(D_a)=D_a\otimes D_a.
\]
%Let $(c_{ij})_{i,j\in \I_{N-1}}$ be the Cartan matrix of $\gl_N$. 

\begin{lemma}\label{lem:inv}
\begin{itemize}
%\item[(1)] If $N$ is even, then there is an algebra automorphism $\vartheta:\U_N\rightarrow \U_N$ such that $E_i\mapsto q^{c_{i,-i}/2}E_{-i}K_i,F_i\mapsto q^{c_{i,\tau i}/2} F_{-i}K_i^{-1},D_a\mapsto D_{-a}^{-1}$.
\item[(1)] There is a Chevalley involution $\omega:\U_N\rightarrow \U_N$ such that $\omega(E_i)=F_i,\omega(F_i)=E_i,\omega(D_a)=D_a^{-1}$.

\item[(2)] There is an anti-involution $\sigma:\U_N\rightarrow \U_N$ such that $\sigma(E_i)=E_i,\sigma(F_i)=F_i,\sigma(D_a)=D_a^{-1}$.

%\item[(2)] If $N$ is odd, then there is an algebra automorphism $\vartheta:\U_N\otimes \Q(q^{1/2})\rightarrow \U_N\otimes \Q(q^{1/2})$ such that $E_i\mapsto q^{c_{i,-i}/2}E_{-i}K_i,F_i\mapsto q^{c_{i,\tau i}/2} F_{-i}K_i^{-1},D_a\mapsto D_{-a}^{-1}$.
\end{itemize}
\end{lemma}

\begin{proof}
One can directly check these maps are well-defined using the above presentation of $\U_N$.
\end{proof}

Let $\cT_{N}$ for $N\in\N$ be the (type A) quantum coordinate ring generated by $t_{i,j},i,j\in  \I_{N}$ subject to the following relations for $i<k,j<l$
\begin{align}\label{def:cT}
t_{ij} t_{kj} = q t_{kj} t_{ij},\quad t_{ij} t_{il} = q t_{il} t_{ij}, \quad t_{il} t_{kj} =t_{kj} t_{il},\quad t_{ij} t_{kl} =t_{kl} t_{ij}+(q-q^{-1})t_{il} t_{kj}.
\end{align}
$\cT_{N}$ is equipped with a coalgebra structure by $\Delta(t_{ij})=\sum_{k} t_{ik}\otimes t_{kj}$.

\subsection{iQuantum groups of type AIII}
\label{sec:iQG}
 The Satake diagram of type AIII$_{2n-1}$ is given below.
\begin{align}\label{diag:odd}
 \begin{tikzpicture}[baseline=0,scale=2.0]
		\node  at (-2.1,0) {$\circ$};
		\node  at (-1.3,0) {$\circ$};
		\node  at (1.3,0) {$\circ$};
		\node  at (2.1,0) {$\circ$};
		\node  at (-0.3,0) {$\circ$};
		\node  at (0.3,0) {$\circ$};
		\node  at (0,0) {$\circ$};
		\draw[-] (-2.05,0) to (-1.35, 0);
		\draw[-] (-1.25,0) to (-1.05, 0);
		\draw[dashed] (-1.05,0) to (-0.35, 0);
		\draw[dashed] (0.35,0) to (1.05, 0);
		\draw[-] (1.05,0) to (1.25, 0);
		\draw[-] (1.35, 0) to (2.05,0);
		\node at (-2.1,-0.2) {\small$-n+1$};
		\node at (-1.3,-0.2) {\small $-n+2$};
		\node at (-0.3,-0.2) {\small$-1$};
		\node at (0.3,-0.2) {\small$1$};
		\node at (1.3,-0.2) {\small$n-2$};
		\node at (2.1,-0.2) {\small$n-1$ };
		\node at (0,-0.2) {\small$0$ };
        \draw[-] (-0.25,0) to (-0.05,0);
        \draw[-] (0.05,0) to (0.25,0);
        \draw[bend left,<->,red] (-1.3,0.1) to (1.3,0.1);
        \draw[bend left,<->,red] (-2.1,0.1) to (2.1,0.1);
        \node at (0,0.6) {$\textcolor{red}{\tau} $};
%\node at (0, -0.55) {AIII$_n,n\geq 4$};
	\end{tikzpicture}
    \end{align}
The iquantum group $\Ui_n$ of type AIII$_{2n-1}$ is the subalgebra of $\U_{2n}$ generated by
\begin{align*}
&B_i= F_i + E_{-i} K_i^{-1}, \qquad \big( i\in [1-n,n-1],i\neq 0\big),
\\
&B_0= F_0 + q^{-1} E_0 K_0^{-1}+K_0^{-1},
\qquad d_r^{\pm1}= \big(D_{r}D_{-r}\big)^{\pm 1},\qquad \big(r\in [\hf,n-\hf]\big).
\end{align*}
Write $k_i=K_i K_{-i}^{-1}$ for $i\in \I_{2n-1}$. The diagram involution $\tau:i\mapsto -i$ induces an algebra involution on $\Ui_n$. The quantum symmetric pair $(\U_{2n}, \Ui_n)$ specializes to the symmetric pair $(\gl_{2n},\gl_n\oplus\gl_{n})$ at $q\rightarrow 1$.

We next consider the symmetric pair of type AIII$_{2n}$; the corresponding Satake diagram is given below.
\begin{align}\label{diag:ev}
 \begin{tikzpicture}[baseline=0,scale=2.0]
		\node  at (-2.1,0) {$\circ$};
		\node  at (-1.3,0) {$\circ$};
		\node  at (1.3,0) {$\circ$};
		\node  at (2.1,0) {$\circ$};
		\node  at (-0.3,0) {$\circ$};
		\node  at (0.3,0) {$\circ$}; 
		\draw[-] (-2.05,0) to (-1.35, 0);
		\draw[-] (-1.25,0) to (-1.05, 0);
		\draw[dashed] (-1.05,0) to (-0.35, 0);
		\draw[dashed] (0.35,0) to (1.05, 0);
		\draw[-] (1.05,0) to (1.25, 0);
		\draw[-] (1.35, 0) to (2.05,0);
		\node at (-2.1,-0.2) {\small$-n+\hf$};
		\node at (-1.3,-0.2) {\small $-n+\frac{3}{2}$};
		\node at (-0.3,-0.2) {\small$-\hf$};
		\node at (0.3,-0.2) {\small$\hf$};
		\node at (1.3,-0.2) {\small$n-\frac{3}{2}$};
		\node at (2.1,-0.2) {\small$n-\hf$ };
        \draw[-] (-0.25,0) to (0.25,0);
        \draw[bend left,<->,red] (-1.3,0.1) to (1.3,0.1);
        \draw[bend left,<->,red] (-2.1,0.1) to (2.1,0.1);
        \node at (0,0.6) {$\textcolor{red}{\tau} $};
%\node at (0, -0.55) {AIII$_n,n\geq 4$};
	\end{tikzpicture}
    \end{align}
The iquantum group $\Uj_n$ of type AIII$_{2n}$ is the subalgebra of $\U_{2n+1}$ generated by
\begin{align*}
&B_i=F_{i}+ E_{-i} K_{i}^{-1},\qquad B_{-i}=F_{-i}+ K_{-i}^{-1} E_{i},\qquad (i\in [\hf,n-\hf]),
\\
&d_0^{\pm1}=D_0^{\pm1},\qquad d_r^{\pm1}=\big(D_r D_{-r}\big)^{\pm1},\qquad (r\in [1,n]).
\end{align*}
Write $k_i=K_{i}K_{-i}^{-1}$ for $i\in \I_{2n}$. The quantum symmetric pair $(\U_{2n+1}, \Uj_n)$ specializes to the symmetric pair $(\gl_{2n+1},\gl_n\oplus\gl_{n+1})$ at $q\rightarrow 1$.
\begin{remark}
Our convention follows \cite{BW18b,BK19} and is different from \cite{LX22}. The relation between iquantum groups here and those used in \cite{LX22} can be understood as follows. Recall $\sigma$ and $\omega$ from Lemma~\ref{lem:inv}.

Let $\RUi_n$ be the subalgebra of $\U_{2n}$ generated by $d_r^{\pm1},r\in [\hf,n-\hf]$ and 
\[
e_i=E_i+F_{-i}K_i^{-1},\quad f_i=E_{-i}+F_{i}K_{-i}^{-1},\quad t_0=E_0+q F_0 K_0^{-1}+K_0^{-1}
\]
for $i\in [1,n-1]$. The algebra $\RUi_n$ is the one considered in \cite{LX22} and the composition $\sigma\omega$ restricts to an algebra anti-isomorphism $\sigma\omega: \Ui_n\rightarrow \RUi_n, B_i\mapsto e_i,B_{-i}\mapsto f_i,t_0\mapsto B_0$ for $i\in [1,n-1]$. 

Similarly, let $\RUj_n$ be the subalgebra of $\U_{2n+1}$ generated by $d_0^{\pm1},d_r^{\pm1},r\in [1,n]$ and 
\[
e_i=E_i+K_i^{-1}F_{-i},\quad f_i=E_{-i}+F_{i}K_{-i}^{-1},\qquad i\in[\hf,n-\hf].
\]
The algebra $\RUj_n$ is the one considered in \cite{LX22} and the composition $\sigma\omega$ restricts to an algebra anti-isomorphism $\sigma\omega: \Uj_n\rightarrow \RUj_n,B_i\mapsto e_{i}, B_{-i}\mapsto f_i$ for $i\in [\hf,n-\hf]$.

Moreover, in both cases, the twisted coproduct $(\sigma\omega\otimes \sigma \omega)\circ \Delta\circ \sigma\omega$ is exactly the one considered in \cite{LX22}.
\end{remark}

%%%%%%%%%%%%
\subsection{Relative braid group symmetries on the modules}
\label{sec:braid}
In the AIII$_{2n-1}$ Satake diagram, the vertex $0$ is fixed by the diagram involution $\tau$. In this case, the idivided powers are introduced in \cite{BW18a,BeW18,CLW21} in theory of icanonical basis. 
We recall the definition of idivided powers $B_{i,\ov{p}}^{(m)}$ for $i=0, m\in \N, \ov{p}\in \Z/2\Z$ as below:
\begin{align}
\label{def:idv}
\begin{split}
&\edvi{m}=\frac{1}{[m]!}
\begin{cases}
B_0 \prod_{r=1}^{k} (B_0^2- [2r]^2),& \text{ if } m=2k+1,
\\
\prod_{r=1}^{k} (B_0^2- [2r-2]^2),& \text{ if } m=2k;
\end{cases}
\\
&\odvi{m}=\frac{1}{[m]!}
\begin{cases}
B_0 \prod_{r=1}^{k} (B_0^2- [2r-1]^2),& \text{ if } m=2k+1,
\\
\prod_{r=1}^{k} (B_0^2- [2r-1]^2),& \text{ if } m=2k.
\end{cases}
\end{split}
\end{align}
For $j\neq 0,k\in \N$, we set $\dv{j}{k}$ to be $\frac{B_j^k}{[k]!}$.

We recall the relative Weyl group from \cite[Section 25]{Lus03}; notations here are following \cite{WZ23}. Let $N\in \N$ and $n=\lfloor N/2\rfloor$. Let $W=\langle s_i|i\in \I_{N-1}\rangle$ be the Weyl group of type $A_{N-1}$. If $N$ is even, we define
\begin{align}
\bs_i=
\begin{cases}
s_i s_{-i}, & \text{ for } 1\le i\le n-1,
\\
s_0, & \text{ for } i=0.
\end{cases}
\end{align}
If $N$ is odd, we define
\begin{align}
\bs_i=
\begin{cases}
s_{i+\hf} s_{-i-\hf}, & \text{ for } 1\le i\le n-1,
\\
s_{\hf} s_{-\hf} s_{\hf}, & \text{ for } i=0.
\end{cases}
\end{align}
The relative Weyl group $W^\circ$ is the subgroup of $W$ generated by $\bs_i,i\in [0,n-1]$. The group $W^\circ$ itself is isomorphic to the Weyl group of type B$_n$ and $\bs_i$ are identified with the simple reflections in this Weyl group of type B$_n$. 

Let $X=\bigoplus_{i\in \I_N} \Z\epsilon_i$ be the weight lattice of $\gl_N$. There is an involution $\tau:X\rightarrow X, \epsilon_i\mapsto-\epsilon_{-i}$, which is compatible with the involution $\tau$ in the underlying Satake diagram. Following \cite{BW18b}, we set 
$X_\imath=X\big/\{\mu+\tau\mu|\mu\in X\}$ and use $\ov{\mu}$ to denote the image of $\mu\in X$ in $X_\imath$. If $N=2n$ ({\em resp.} $N=2n+1$), $X_i$ is the set of $\Ui_n$-weights ({\em resp.} $\Uj_n$-weights).

Let $V$ be a $X_\imath$-weighted (right) $\Ui_n$-module ({\em resp.} $\Uj_n$-module) with a decomposition $M=\bigoplus_{\ov{\mu}\in X_\imath} M_{\ov{\mu}}$; cf. \cite[Section 3.3]{BW18b}. %set $N=2n$ ({\em resp.} $N=2n+1$). %For any $\ov{\mu}\in X_\imath$, we define the iweight space 
%\begin{align} 
%M_{\ov{\mu}}=\{v\in M| v k_i = q^{\mu_{i-\hf}-\mu_{i+\hf}+\mu_{-i+\hf}-\mu_{-i-\hf}}v,\forall i\in \I_{N-1}\}.
%\end{align}
Following \cite[Definition 2.8]{WZ25}, a $\Ui_n$-module $V$ is integrable if, for any $\mu=\sum_{r\in\I_{N}} \mu_r \epsilon_r\in X$, each vector $v\in V_{\ov{\mu}}$ is annihilated by $B_j^{(k)}$ ($j\in \I_{2n-1},j\neq 0$) and $B_{0,\ov{\mu_{-\hf}+\mu_{\hf}+1}}^{(k)}$ for $k\gg 0$; $V$ is an integrable $\Uj_n$-module if each vector $v\in V$ is annihilated by $B_j^{(k)}$ ($j\in \I_{2n}$) for $k\gg0$. 

Let 
\begin{align}\label{def:rho}
\rho=
\begin{cases}
\sum_{i\in [\hf,n-\hf]} \epsilon_i, &\text{ if } N=2n \text{ is even},\\
\sum_{i\in [1,n]} \epsilon_i, &\text{ if } N=2n+1 \text{ is odd}.
\end{cases}
\end{align}
Note that $\ov{\rho}\in X_\imath$ is invariant under the action of $W^\circ$.

\begin{proposition}[\text{\cite[(7.2), (7.5), (7.8), Theorem E, Lemma 7.4]{WZ25}}]
\label{prop:braidmod}
\begin{itemize}
\item[(1)] Let $V$ be an integrable right $\Ui_n$-module. There exist symmetries $\TT_i,i\in [0,n-1]$ on $V$, which satisfy braid relations in $W^\circ$. Moreover, for any $v\in V_{\ov{\rho}}$, we have
\begin{align}\label{eq:Tmodi}
v\TT_i =v\sum_{p=0}^\infty (-q)^{-p} B_i^{(p)}B_{-i}^{(p)} \quad(i>0),
\qquad
v\TT_0 =v\sum_{p=0}^\infty (-q)^{-p} \edvi{2p}.
\end{align} 

\item[(2)] Let $V$ be an integrable right $\Uj_n$-module. There exist symmetries $\TT_i,i\in[0,n-1]$ on $V$, which satisfy braid relations in $W^\circ$. Moreover, for any $v\in V_{\ov{\rho}}$, we have
\begin{align}
v\TT_i &=v\sum_{p=0}^\infty (-q)^{-p} B_{i+\hf}^{(p)}B_{-i-\hf}^{(p)} \qquad(i>0),
\\
v\TT_0 &=v\sum_{t,l\ge 0}  (-1)^{t+l}q^{-\frac{(t-l)(t-l+1)}{2}-t-l} 
B_{-\hf}^{(t)} B_{\hf}^{(t+l)}  B_{-\hf}^{(l)}.\label{eq:TT0}
\end{align}

\item[(3)] For any $\ov{\mu}\in X_\imath$, the symmetry $\TT_i$ sends $V_{\ov{\mu}}$ to $V_{\ov{\bs_i\mu}}$. In particular, $V_{\ov{\rho}}$ is invariant under $\TT_i$ for any $i\in [0,n-1]$.

\end{itemize}
\end{proposition}

\begin{remark}
The symmetry $\TT_i$ here corresponds to $\dot{\TT}'_{i,-1}$ in \cite[Section 7.2]{WZ25}.
\end{remark}

\subsection{Double centralizer property}
 Let $R_1,R_2$ be two algebras and $V$ be a $(R_1,R_2)$-bimodule. Denote by $\Phi:R_1\rightarrow \End(V)$ the action of $R_1$ on $V$ and by $\Psi:R_2\rightarrow \End(V)$  the action of $R_2$ on $V$.
\begin{align*}
R_1 \curvearrowright V \curvearrowleft R_2
\end{align*}
 We say that these two actions satisfy the {\em double centralizer property} if
 \[
 \End_{\Phi(R_1)}(V)=\Psi(R_2), \qquad \End_{\Psi(R_2)}(V)=\Phi(R_1).
 \]
 We sometimes write $\End_{ R_1}(V)$ for $\End_{\Phi(R_1)}(V)$.
 If the actions of $R_1$ and $R_2$ satisfy the double centralizer property, then $V$ can be uniquely decomposed into the following form:
\[
V=\bigoplus_i U_i\otimes V_i,
\]
 where $U_i$ are pairwise non-isomorphic irreducible $R_1$-modules and $V_i$ are pairwise non-isomorphic irreducible $R_2$-modules.
 
\subsection{Type B Hecke algebra and iSchur duality}\label{sec:ischur}
We recall the formulation of iSchur duality from \cite{BW18a}.

Let $\cH_B(n)$ be the Hecke algebra of type B$_n$ over $\Q(q)$ with generators $\sigma_i,0\le i\le n-1$ subject to 
\begin{align*}
&\sigma_i \sigma_{i+1} \sigma_i = \sigma_{i+1} \sigma_i \sigma_{i+1},
\qquad 1\le i\le n-2,
\\
& \sigma_0 \sigma_1 \sigma_0 \sigma_1= \sigma_1\sigma_0 \sigma_1 \sigma_0,\qquad
\sigma_i \sigma_j = \sigma_j \sigma_i,\qquad |i-j|>1,
\\
&(\sigma_i+q )(\sigma_i-q^{-1})=0,\qquad 0\le i\le n-1.
\end{align*}
The subalgebra of $\cH_B(n)$ generated by $\sigma_i,1\le i\le n-1$ is isomorphic to the Hecke algebra $\cH_A(n-1)$ of type $A_{n-1}$. 

Recall that $W^\circ$ is isomorphic to the Weyl group of type B$_n$ with simple reflections $\bs_i,i\in [0,n-1]$. For any $w=\bs_{i_1}\cdots\bs_{i_k}$, there is a well-defined element $\sigma_w=\sigma_{i_1}\cdots\sigma_{i_k}$.

Let $V=\bC^M$ be the natural representation of $\U(\gl_M)$ with the standard basis $\{v_i|i\in \I_M\}$. i.e.,
\begin{align}\label{eq:natural}
E_i v_r=\delta_{i+\hf,r} v_{r-1}, \qquad F_i v_r=\delta_{i-\hf,r} v_{r+1}, 
\qquad D_s v_r=q^{\delta_{s,r} }v_r.
\end{align}
For $\lambda=(\lambda_1,\cdots,\lambda_n)\in \I_{M}^n$, set
$
v_\lambda=v_{\lambda_1} \otimes v_{\lambda_2} \otimes \cdots \otimes v_{\lambda_n}.
$
It is clear that the set $\{v_\lambda|\lambda\in \I_{M}^n\}$ forms a basis of $V^{\otimes n}$.
Via restriction, the $\U(\gl_M)$-module $V^{\otimes n}$ can be viewed as a $\Ui_m$-module (when $M=2m$ is even) or a $\Uj_m$-module (when $M=2m+1$ is odd).

On the other hand, there is a well-known right $\cH_B(n)$-action on $V^{\otimes n}$ given as follows (cf. \cite{BW18a,LNX22}): for $\lambda=(\lambda_1,\ldots,\lambda_n)$,
\begin{align}\label{eq:sigma}
v_\lambda \sigma_i=
\begin{cases}
v_{\lambda \cdot \sigma_i}, 
& \text{ if } i\neq 0, \lambda_i<\lambda_{i+1} \text{ or } i=0,\lambda_1>0;
\\
q^{-1}v_{\lambda \cdot \sigma_i}, 
& \text{ if } i\neq 0, \lambda_i=\lambda_{i+1} \text{ or } i=0,\lambda_1=0;
\\
v_{\lambda \cdot \sigma_i}+(q^{-1}-q) v_\lambda, & \text{ if } i\neq 0, \lambda_i>\lambda_{i+1} \text{ or } i=0,\lambda_1<0.
\end{cases}
\end{align}
where 
\begin{align}
\lambda \cdot \sigma_i=
\begin{cases}
(\lambda_1,\ldots,\lambda_{i-1},\lambda_{i+1},\lambda_i,\lambda_{i+2},\ldots,\lambda_n),
& \text{ if } i\neq 0;
\\
(-\lambda_1,\lambda_2,\ldots,\lambda_n), & \text{ if } i=0.
\end{cases}
\end{align}

\begin{proposition}\label{prop:ischur}
\cite{BW18a}
If $M=2m$, then the above left $\Ui_m$-action and right $\cH_B(n)$-action on $V^{\otimes n}$
\begin{align}\label{eq:SchurBi}
\Ui_m \curvearrowright V^{\otimes n} \curvearrowleft \cH_B(n)
\end{align}
satisfy the double centralizer property.

If $M=2m+1$, then the above left $\Uj_m$-action and right $\cH_B(n)$-action on $V^{\otimes n}$
\begin{align}\label{eq:SchurBj}
\Uj_m \curvearrowright V^{\otimes n} \curvearrowleft \cH_B(n)
\end{align}
satisfy the double centralizer property.
\end{proposition}

\section{$\mathrm{i}$Howe dual pair $(\Ui_m,\Ui_n)$ and $\mathrm{i}$Schur duality}\label{sec:iHowe}

In this section, we fix $m,n\in \N$ and set $N=2n,M=2m$. We show that the iSchur duality \eqref{eq:SchurBi} between $\Ui_m$ and $\cH_B(n)$ can be recovered from the iHowe dual pair $(\Ui_m,\Ui_n)$ in \cite{LX22}. 

\subsection{$(\Ui_m,\Ui_n)$-iHowe duality}\label{sec:Uiduality}
We set up the iHowe duality following \cite{LX22}.  
Set
\begin{align*}
\begin{split}
\Xi_{m|n,d}^{\imath}&:=\{(a_{ij})_{i\in [\frac{1}{2}-m,m-\frac{1}{2}],j\in  [\frac{1}{2}-n,n-\frac{1}{2}]}\in \Mat_{2m\times 2n}(\N)|a_{i,j}=a_{-i,-j},\sum a_{i,j}=2d\},
\\
\Xi_{m|n}^\imath&:=\bigsqcup_{d\in \N} \Xi_{m|n,d}^{\imath}.
\end{split}
\end{align*}

Define $\mathcal{I}$ to be the right ideal of $\mathcal{T}_{2n}$ (see \eqref{def:cT} for its definition) generated by
\begin{align*}
t_{ij}-t_{-i,-j}-(q-q^{-1})t_{i,-j} ,\qquad t_{i,-j}-t_{-i,j},\qquad i,j\in [\frac{1}{2},n-\frac{1}{2}],
\end{align*}
and define $\mathcal{T}_n^\imath:=\mathcal{T}_{2n} / \mathcal{I}$. Let $>$ be lexicographic order on $\I_N^2$. $\mathcal{T}_n^\imath$ has a coalgebra structure but not an algebra structure. Denote $t^{(A)}:=\prod^{<}_{(i,j)>(0,0)} (t_{ij})^{a_{ij}}$ for $A\in \Xi_{m|n}^\imath$, where the product is arranged such that $t_{ij}$ is in front of $t_{kl}$ if $(i,j)<(k,l)$.

The space $\cV_{m|n}^\imath$ is defined to be the subspace of $\mathcal{T}_{\max(m,n)}^\imath$ spanned by $\{t^{(A)}|A\in \Xi_{m|n}^\imath\}$, and $\cV_{m|n,d}^\imath$ is defined to be the subspace of $\mathcal{T}_{\max(m,n)}^\imath$ spanned by $\{t^{(A)}|A\in \Xi_{m|n,d}^\imath\}$. By \cite[Theorem 5.3]{LX22}, $\{t^{(A)}|A\in \Xi_{m|n,d}^\imath\}$ is a basis for $\cV_{m|n,d}^\imath$ and $\cV_{m|n}^\imath=\bigoplus_{d\in \N} \cV_{m|n,d}^\imath$.
%The space $\mathcal{V}_{m|n}^\imath$ admits a $(\Ui_m,\Ui_n)$-module structure and the actions are explicitly given in \cite[Proposition 5.5]{LX22}.

Let $E_{i,j}$ be the elementary matrix in $\Mat_{2m\times 2n}(\N)$ and $E_{i,j}^\theta:=E_{i,j}+E_{-i,-j}$ for $i\in \I_{2m},j\in \I_{2n}$. Clearly, $E_{i,j}^\theta=E_{-i,-j}^\theta$.

\begin{proposition}[\text{cf. \cite[Proposition 5.5, Theorem 4.3]{LX22}}]
\label{prop:BtA}
Let $d\in \N$. There is a (left) $\Ui_m$-action on $\cV_{m|n,d}^\imath$ explicitly given as follows: for $i\in [1-m,m-1], i\neq 0, A\in \Xi_{m|n,d}^\imath$, 
\begin{align}
B_i t^{(A)} = &\sum_{j\in [\hf-n,n-\hf]} q^{\sum_{k\in[j,n-\hf]}
\big(a_{i+\hf,k}-  a_{i-\hf,k}\big)+1}  [a_{i-\hf,j}]
 t^{(A +E_{i + \hf,j}^\theta-E_{i - \hf,j}^\theta)},
 \label{eq:BitA1}
 \\\notag
 B_0 t^{(A)} = &\sum_{j\in [\hf-n,n-\hf]} q^{\sum_{k\in[j,n-\hf]}\big(a_{\hf,k}-  a_{-\frac{1}{2},k}\big)+\delta_{j>0}}  [a_{-\hf,j}]
 t^{(A +E_{\hf,j}^\theta-E_{-\hf,j}^\theta)}
 \\
 &\qquad\qquad+q^{\sum_{j\in[\hf,n-\hf]} \big(a_{\frac{1}{2},j}-a_{-\hf,j}\big)}t^{(A)}.
 \label{eq:B0tA1}
\end{align}
There is a (right) $\Ui_n$-action on $\cV_{m|n,d}^\imath$ explicitly given as follows: for $i\in [1-n,n-1], i\neq 0, A\in \Xi_{m|n,d}^\imath$, 
\begin{align}
t^{(A)} B_i = &\sum_{j\in [\frac{1}{2}-m,m-\frac{1}{2}]} q^{\sum_{k\in[j+1,m-\hf]}
\big(a_{k,i+\hf}-  a_{k,i-\hf}\big)}  [a_{j,i+\frac{1}{2}}]
 t^{(A +E_{j,i-\frac{1}{2}}^\theta-E_{j,i +\frac{1}{2}}^\theta)},
 \label{eq:BitA2}
 \\\notag
 t^{(A)} B_0 = &\sum_{j\in [\frac{1}{2}-m,m-\frac{1}{2}]} q^{\sum_{k\in[j+1,m-\frac{1}{2}]}\big(a_{k,\frac{1}{2}}-  a_{k,-\frac{1}{2}}\big)+\delta_{j<0}}  [a_{j,\frac{1}{2}}]
 t^{(A +E_{j,- \frac{1}{2}}^\theta-E_{j,\frac{1}{2}}^\theta)}
 \\
 &\qquad\qquad+q^{\sum_{j\in[\frac{1}{2},m-\frac{1}{2}]} \big( a_{j,\frac{1}{2}}-a_{j,-\frac{1}{2}}\big)}t^{(A)}.
 \label{eq:B0tA2}
\end{align}
The above two actions
\begin{align}\label{eq:Uiduality}
\Ui_m \curvearrowright \cV_{m|n,d}^\imath \curvearrowleft \Ui_n
\end{align}
satisfy the double centralizer property.
\end{proposition}

\begin{remark}
In \cite{LX22}, they parameterize rows ({\em resp.} columns) of a $2m\times 2n$ matrix by $[-m,-1]\cup [1,m]$ ({\em resp.} $[-n,-1]\cup [1,n]$), which are different from ours. This leads to the shift of indices between formulas in Proposition~\ref{prop:BtA} and \cite[Proposition 5.5]{LX22}.
\end{remark}

The $\Ui_n$-weight space $\big(\cV_{m|n,d}^{\imath}\big)_{\ov{\mu}}$ for $\mu=\sum_{r\in \I_N}\mu_r\epsilon_r$ is given by
\[
\big(\cV_{m|n,d}^{\imath}\big)_{\ov{\mu}}
=\Big\{v\in \cV_{m|n,d}^{\imath}\Big| v d_r=q^{\mu_r+\mu_{-r}}v,\forall r\in [\hf,n-\hf]\Big\}.
\]
Then $\cV_{m|n,d}^{\imath}=\bigoplus_{\ov{\mu}\in X_\imath} \big(\cV_{m|n,d}^{\imath}\big)_{\ov{\mu}}$. Note that the value of $\mu_r+\mu_{-r}$ does not depend on the choice of a preimage of $\ov{\mu}$ in $X$.

\subsection{The iweight $\ov{\rho}$ space}

We define $\cV_{m|n,d}^{\imath,0}$ by
\begin{align}\label{def:0iwt}
\cV_{m|n,d}^{\imath,0}:=\{v\in\cV_{m|n,d}^{\imath}|v k_i=v, \forall i\in [1,n-1]\}.
\end{align}

Recall that $k_i=d_{i-\hf}d_{i+\hf}^{-1}$. The actions of $d_r, k_{i}$ for $i\in [1,n-1],r\in[\hf,n-\hf]$ on $\cV_{m|n,d}^{\imath}$ are given by
\begin{align}\label{eq:ktA}
t^{(A)} d_{r} = q^{\sum_{k\in \I_{2m}} a_{k,r}  }t^{(A)},
\quad
t^{(A)} k_{i} = q^{\sum_{k\in \I_{2m}} (a_{k,i-\hf}-a_{k,i+\hf}) }t^{(A)},\quad
\forall A\in \Xi_{m|n,d}^{\imath}.
\end{align}

Due to \eqref{eq:ktA}, the element $t^{(A)}$ lies in $\mathcal{V}_{m|n,d}^{\imath,0}$ if and only if $\sum_i a_{i,j}=\sum_i a_{i,j'}$ for any two columns $j,j'$ of $A$. In particular, the space $\mathcal{V}_{m|n,n}^{\imath,0}$ admits a basis $\{t^{(A)}| A\in \Xi_{m|n,n}^{\imath,0}\}$ where the set $\Xi_{m|n,n}^{\imath,0}$ is given by
\begin{align}
\Xi_{m|n,n}^{\imath,0}&:=\big\{A=(a_{ij}) \in\Xi_{m|n,n}^{\imath} \big|\sum_i a_{i,j}=1,\forall j\in  [\frac{1}{2},n-\frac{1}{2}]\big\}.
\end{align}
Note that $\cV_{m|n,n}^{\imath,0}$ can be identified with the $\Ui_n$-weight space $ \big(\cV_{m|n,n}^{\imath}\big)_{\ov{\rho}}$.

Let $V=\bC^M$ be the natural representation of $\U_M$; $V$ can be viewed as a $\Ui_m$-module via restriction.

\begin{theorem}\label{thm:0wti}
 There is a natural isomorphism of $\Ui_m$-modules $\cV_{m|n,n}^{\imath,0}\cong V^{\otimes n}$.
\end{theorem}

\begin{proof}
Let $\{ v_i|i\in \I_{2m} \}$ be the natural basis of $V$; see \eqref{eq:natural}. For $\lambda=(\lambda_1,\cdots,\lambda_n)\in \I_{2m}^n$, set
$v_\lambda=v_{\lambda_1} \otimes v_{\lambda_{2}} \otimes \cdots \otimes v_{\lambda_n}.$
The set $\{v_\lambda|\lambda\in \I_{2m}^n\}$ forms a basis of $V^{\otimes n}$.

Denote by $e_i=(0,\cdots,0,1,0\cdots,0)^t$ the $M\times 1$ column vector. Then we have the following bijection between bases of $ V^{\otimes n}$ and $\cV_{m|n,n}^{\imath,0}$  
\begin{align*}
 \I_{2m}^n &\rightarrow \Xi_{m|n,n}^{\imath,0}
\\
\lambda & \mapsto A_\lambda:=(e_{-\lambda_{n}},\cdots, e_{-\lambda_1},e_{\lambda_1},\cdots, e_{\lambda_n}).
\end{align*}
Hence, we have a linear isomorphism $\iPhi:V^{\otimes n} \cong\cV_{m|n,n}^{\imath,0} , v_\lambda \mapsto t^{(A_\lambda)}$.

We next show that $\iPhi$ is $\Ui_m$-equivariant. Write $A_\lambda=(a_{ij}^\lambda)$. Recall that $D_r v_{\lambda_j}=q^{a_{r,j-\hf}^\lambda} v_{\lambda_j}$ for $r\in \I_{2m}$ and hence $K_i v_{\lambda_j}=q^{a_{i-\hf,j-\hf}^\lambda-a_{i+\hf,j-\hf}^\lambda} v_{\lambda_j}$ for $i\in \I_{2m-1}$. Using these formulas, it is direct to show that $\iPhi(d_r v_\lambda)=d_r\iPhi( v_\lambda)$ for $r\in [\hf,n-\hf]$.
 
 Let $i\in \I_{2m-1}\setminus\{0\}$. We next show that $\iPhi(B_i v_\lambda)=B_i\iPhi( v_\lambda)$. Recall that $\Delta(B_i)=B_i\otimes K_i^{-1} + 1\otimes F_i + k_i^{-1}\otimes E_{-i} K_i^{-1} $ and $\Delta(k_i)=k_i\otimes k_i$. Let $\Delta^{{k}}$ be the algebra homomorphism $\U\rightarrow \underbrace{\U\otimes \cdots \otimes\U}_{k+1}$ induced from the coproduct $\Delta$. By induction, we have
 \begin{align}\notag
 \Delta^{{n-1}}(B_i)=&\sum_{r=1}^{n} \underbrace{1\otimes \cdots \otimes 1}_{r-1} \otimes F_i \otimes \underbrace{K_i^{-1}\otimes \cdots \otimes K_i^{-1}}_{n-r}
 \\
 &+\sum_{r=1}^{n} \underbrace{k_i^{-1} \otimes \cdots \otimes k_i^{-1}}_{r-1} \otimes E_{-i} K_i^{-1}\otimes\underbrace{ K_i^{-1} \otimes \cdots \otimes K_i^{-1}}_{n-r}.
 \label{eq:Deln}
 \end{align}

 Using $\eqref{eq:Deln}$, we have for $i\neq 0$
 \begin{align}
 B_i v_{\lambda}=&\sum_{j=1}^n q^{\sum_{k>j}a_{i+\frac{1}{2},k-\frac{1}{2}}^\lambda -\sum_{k>j}a_{i-\frac{1}{2},k-\frac{1}{2}}^\lambda}  \delta_{\lambda_j=i-\frac{1}{2}} v_{\lambda+\epsilon^j}
 \\\notag
 &+ q^{\sum_{k=1}^n \big(a_{i+\frac{1}{2},k-\frac{1}{2}}^\lambda -a_{i-\frac{1}{2},k-\frac{1}{2}}^\lambda\big)}  \sum_{j=1}^n
 q^{\sum_{1\leq k<j}\big(a_{-i-\frac{1}{2},k-\frac{1}{2}}^\lambda - a_{-i+\frac{1}{2},k-\frac{1}{2}}^\lambda\big)} \delta_{\lambda_j=-i+\frac{1}{2}} v_{\lambda-\epsilon^j},
 \end{align}
 where $\epsilon^j=(0,\cdots,0,1,0,\cdots ,0)\in \N^n$.

 For $A_\lambda\in \Xi_{m|n,n}^{\imath,0}$, we have $ \delta_{\lambda_j=i-\hf}=[a_{i-\hf,j-\hf}^\lambda]$ and $\delta_{\lambda_j=-i+\hf}=[a_{-i+\hf,j-\hf}^\lambda]$. Hence, we can rewrite the above formula as
  \begin{align*}
 B_i v_{\lambda}=&\sum_{j=1}^n q^{\sum_{k>j}a_{i+\hf,k-\hf}^\lambda -\sum_{k>j}a_{i-\hf,k-\hf}^\lambda}  [a_{i-\hf,j-\hf}^\lambda] v_{\lambda+\epsilon^j}
 \\
 &+ q^{\sum_{k=1}^n \big(a_{i+\frac{1}{2},k-\frac{1}{2}}^\lambda -a_{i-\frac{1}{2},k-\frac{1}{2}}^\lambda\big)}  \sum_{j=-n}^{-1}
 q^{\sum_{ j< k< 0}\big(a_{i+\hf,k+\hf}^\lambda - a_{i-\hf,k+\hf}^\lambda\big)} [a_{i-\hf,j+\hf}^\lambda] v_{\lambda-\epsilon^{-j}}
 \\
 =&\sum_{j=1}^n q^{\sum_{k> j}\big(a_{i+\hf,k-\hf}^\lambda -a_{i-\hf,k-\hf}^\lambda\big)}  [a_{i-\hf,j-\hf}^\lambda] v_{\lambda+\epsilon^j}
 \\
 &+  \sum_{j=-n+1}^{0}
 q^{\sum_{ k>j}\big(a_{i+\hf,k-\hf}^\lambda - a_{i-\frac{1}{2},k-\frac{1}{2}}^\lambda\big)} [a_{i-\hf,j-\hf}^\lambda] v_{\lambda-\epsilon^{-j+1}}.
 \end{align*}
 Note that if $a_{i-\hf,j-\hf}^\lambda\neq 0$, then $a_{i+\hf,j-\hf}^\lambda=0$ and $a_{i+\hf,j-\hf}^\lambda-a_{i-\hf,j-\hf}^\lambda=-1$.
 By definition, $\iPhi(v_{\lambda\pm \epsilon^j})=t^{\big(A_\lambda+E_{\lambda_j\pm 1,j-\frac{1}{2}}^\theta-E_{\lambda_j,j-\frac{1}{2}}^\theta\big)}$ and $E_{k,l}^\theta=E_{-k,-l}^\theta$.
 Thus, we obtain the following formula for $\iPhi(B_i v_\lambda)$
 \begin{align}\notag
 \iPhi(B_i v_\lambda)=&\sum_{j=1}^n q^{\sum_{k>j}\big(a_{i+\frac{1}{2},k-\frac{1}{2}}^\lambda -a_{i-\frac{1}{2},k-\frac{1}{2}}^\lambda\big)}  [a_{i-\hf,j-\hf}^\lambda]
 t^{\big(A_\lambda+E_{i+\hf,j-\hf}^\theta-E_{i-\hf,j-\hf}^\theta\big)}
 \\\notag
 &+   \sum_{j=-n+1}^{0}
 q^{\sum_{ k>j}\big(a_{i+\hf,k-\hf}^\lambda - a_{i-\hf,k-\hf}^\lambda\big)}
 [a_{i-\frac{1}{2},j-\frac{1}{2}}^\lambda] 
 t^{\big(A_\lambda+E_{-i-\hf,-j+\hf}^\theta-E_{-i+\hf,-j +\hf}^\theta\big)}
 \\\notag
 =&\sum_{j\in [\hf-n,n-\frac{1}{2}]}q^{\sum_{k\in[j+1,n-\hf]}\big(a_{i+\hf,k}^\lambda -  a_{i-\hf,k}^\lambda\big)}  [a_{i-\frac{1}{2},j}^\lambda]
 t^{\big(A_\lambda+E_{i+\hf,j}^\theta-E_{i-\hf,j}^\theta\big)}
 \\
 =&\sum_{j\in [\hf-n,n-\frac{1}{2}]}q^{\sum_{k\in[j,n-\hf]}\big(a_{i+\hf,k}^\lambda -  a_{i-\hf,k}^\lambda\big)+1}  [a_{i-\frac{1}{2},j}^\lambda]
 t^{\big(A_\lambda+E_{i+\hf,j}^\theta-E_{i-\hf,j}^\theta\big)}.\label{eq:PhiBv}
 \end{align}
 Comparing \eqref{eq:BitA1} and \eqref{eq:PhiBv}, it is clear that $\iPhi(B_i v_\lambda)=B_i t^{(A_\lambda)}=B_i \iPhi(v_\lambda)$.

 It remains to show that $\iPhi(B_0 v_\lambda)=B_0 \iPhi(v_\lambda)$. Write $B'_0=B_0-K_0^{-1}$. Recall that $\Delta(B'_0)=B'_0\otimes K_0^{-1}+1\otimes B'_0 $. By induction, we have
 \begin{align}
 \Delta^{n-1}(B'_0)= \sum_{r=1}^n \underbrace{1\otimes \cdots \otimes 1}_{r-1} \otimes B'_0\otimes \underbrace{K_0^{-1}\otimes \cdots \otimes K_0^{-1}}_{n-r}.
 \end{align}
 As a consequence, we have
 \begin{align}\notag
 B_0 v_\lambda=&\sum_{j=1}^n q^{\sum_{k>j}a_{\frac{1}{2},k-\frac{1}{2}}^\lambda -\sum_{k>j}a_{-\frac{1}{2},k-\frac{1}{2}}^\lambda}
   (\delta_{\lambda_j=\frac{1}{2}} v_{\lambda-\epsilon^j} +  \delta_{\lambda_j=-\frac{1}{2}} v_{\lambda+\epsilon^j})
   \\
 &\qquad\qquad+q^{\sum_{j\in[\frac{1}{2},n-\frac{1}{2}]} \big(\delta_{\lambda_j=\frac{1}{2}} -\delta_{\lambda_j=-\frac{1}{2}} \big)}v_\lambda.
 \end{align}
 Using arguments similar to the case $B_i,i\neq 0$, we can simplify this relation and show that it is compatible with \eqref{eq:B0tA1}.

 Therefore, we have proved that $\iPhi$ is an isomorphism of $\Ui_m$-modules as desired.
\end{proof}

\subsection{Relative braid group symmetry $\TT_i,i>0$}
In this section, we establish in Theorem~\ref{thm:Uniso} a compatibility between the $(\Ui_m,\Ui_n)$-duality and the $(\U(\gl_{M}),\U(\sl_{n}))$-duality. This theorem implies that the module $\cV_{m|n,n}^{\imath}$ is integrable; it further allows us to identify $\TT_i,i>0$ with Lusztig's braid group symmetries and the Hecke relation for $\TT_i,i>0$ follows. 

In this subsection, we use $[1,n-1]$ to label the simple roots of $\U(\sl_n)$; note that this is different from the convention in Section~\ref{sec:QG}. We label the simple roots of $\U(\gl_{M})$ in the same way as Section~\ref{sec:QG}.

We first recall from \cite{Zha02,LX22} the quantum Howe duality for $(\U(\gl_{M}),\U(\sl_{n}))$. Recall the space $\cT_{P}$ from \eqref{def:cT}. Set
\begin{align}
\begin{split}
\Xi_{M|n,d} &:=\{(a_{ij})_{i\in [\hf-m,m-\hf],j\in [\hf,n-\hf]}\in \Mat_{M\times n}(\N)|\sum a_{i,j}=d\},
\\
\Xi_{M|n}&:=\bigsqcup_{d\in \N} \Xi_{M|n,d}.
\end{split}
\end{align}
For $A=(a_{ij})_{i\in \I_M,j\in [\hf,n-\hf]}\in \Xi_{M|n}$, define $t^{(A)}=\prod^{<}_{i\in \I_M,j\in [\hf,n-\hf]} t_{ij}^{a_{ij}}$. The space $\cV_{M|n}$ is defined to be the subspace of $\mathcal{T}_{\max(M,n)}$ spanned by $\{t^{(A)}|A\in \Xi_{M|n}\}$, and $\cV_{M|n,d}$ is defined to be the subspace of $\mathcal{T}_{\max(M,n)}$ spanned by $\{t^{(A)}|A\in \Xi_{M|n,d}\}$. It is well-known that $\{t^{(A)}|A\in \Xi_{M|n,d}\}$ forms a basis of $\cV_{M|n,d}$; cf. \cite{PW91}.

\begin{proposition}[\text{cf. \cite[Proposition 3.6]{LX22}}]
\label{prop:EFtA}
There exist a (left) $\U(\gl_M)$-action and a (right) $\U(\sl_n)$-action on $\cV_{M|n,d}$ 
\begin{align*}
\U(\gl_M) \curvearrowright \mathcal{V}_{M|n,d} \curvearrowleft \U(\sl_n)
\end{align*}
such that these two actions satisfy the double centralizer property.
Explicitly, the (left) $\U(\gl_M)$-action is given as follows: for $i\in [1-m,m-1]$,
\begin{align}
&E_i t^{(A)}
=\sum_{j\in[\hf,n-\hf]} q^{\sum_{k\in[\hf,j]}\big(a_{i-\hf,k}-  a_{i+\hf,k}\big)+1}  [a_{i+\hf,j}] t^{(A +E_{i- \frac{1}{2},j}-E_{i + \frac{1}{2},j})},
 \\
&F_i t^{(A)}=\sum_{j\in[\hf,n-\hf]}
q^{\sum_{k\in[j,n-\hf]}\big(a_{i+\hf,k} - a_{i-\hf,k}\big)+1}  [a_{i-\hf,j}]
 t^{(A +E_{i+\hf,j}-E_{i -\hf,j})}.
\end{align}

The (right) $\U(\sl_n)$-action on $\cV_{M|n}$ is given as follows: for $i\in [1,n-1]$,
\begin{align}
t^{(A)} E_i = &\sum_{j\in [\hf-m,m-\hf]} q^{\sum_{k\in[\hf-m,j-1]}\big(a_{k,i-\hf}-  a_{k,i+\hf}\big)}  [a_{j,i-\hf}]
 t^{(A +E_{j,i+\frac{1}{2}} -E_{j,i-\frac{1}{2}} )},
 \\
t^{(A)} F_i = &\sum_{j\in [\hf-m,m-\hf]} q^{\sum_{k\in[j+1,m-\hf]}\big(a_{k,i+\hf}-  a_{k,i-\hf}\big)}  [a_{j,i+\frac{1}{2}}]
 t^{(A +E_{j,i-\hf} -E_{j,i+\hf} )}.
\end{align}
\end{proposition}

It is clear from \eqref{diag:odd} that the type AIII$_{2n-1}$ Satake diagram contains a diagonal-type subdiagram $A_{n-1}\times A_{n-1}$ (removing the vertex $0$). Equivalently, there is an algebra embedding
\begin{align}\label{eq:iotan}
\iota_n:\U(\sl_n)\hookrightarrow \Ui_n,\quad  F_i \mapsto B_i,\quad E_i \mapsto B_{-i}, \quad K_i \mapsto k_{i},\quad (i\in[1,n-1]);
\end{align}
cf. \cite[Theorem 7.4]{Let02}. The $\Ui_n$-module $\mathcal{V}_{m|n,d}^\imath$ admit a $\U(\sl_n)$-module structure via restriction.

Note that there is a bijection of sets $\Omega:\Xi_{m|n,d}^{\imath}\rightarrow \Xi_{M|n,d}$ for $M=2m$ by omitting the first $n$ columns. By definition, $\Omega(E_{i,j}^\theta)=E_{i,j}$ for $j>0$.

\begin{theorem}\label{thm:Uniso}
There is an isomorphism of $\big(\Ui_m,\U(\sl_n)\big)$-bimodules for any $d\in \N$
\begin{align}
\tOmega:\cV_{m|n,d}^\imath\cong\cV_{M|n,d}, \qquad t^{(A)}\mapsto t^{(\Omega A)}.
\end{align}
In other words, we have the following commutative diagram
%\begin{align*}
%\U(\sl_M) \curvearrowright \mathcal{V}_{M|n,d} \curvearrowleft \U(\sl_n)
%\Ui_m \curvearrowright \mathcal{V}_{m|n,d}^\imath \curvearrowleft \Ui_n
%\end{align*}
\begin{center}
\begin{tikzcd}
  & \U(\gl_M)
  & \curvearrowright\qquad  \cV_{M|n,d} \qquad \curvearrowleft
  &  \arrow[hookrightarrow]{d}{\iota_n}\U(\sl_n) \\
  & \Ui_m \arrow[hookrightarrow]{u}
  & \curvearrowright\qquad  \cV_{m|n,d}^\imath \arrow{u}{\tOmega}\qquad \curvearrowleft
  &\Ui_n
\end{tikzcd}
\end{center}
\end{theorem}

\begin{proof}
Note that $\Omega$ induces a bijection between a basis of $\cV_{m|n,d}^\imath$ and a basis of $\cV_{M|n,d}$; hence, $\tOmega$ is a linear isomorphism.
It suffices to show that $\tOmega$ is a right $\U(\sl_n)$-module homomorphism and a left $\Ui_m$-module homomorphism. 

By Proposition~\ref{prop:BtA} and Proposition~\ref{prop:EFtA}, it is clear that $\tOmega( t^{(A)}B_i)= t^{(\Omega A)}F_i$ for $i>0$. We show that $\tOmega( t^{(A)}B_{-i})= t^{(\Omega A)}E_i$ for $i>0$. By Propositions~\ref{prop:BtA}, we have
\begin{align*}
\tOmega( t^{(A)}B_{-i})
 = &\sum_{j\in [\hf-m,m-\hf]} q^{\sum_{k\in[j+1,m-\hf]}\big(a_{k,-i+\hf}-  a_{k,-i-\hf}\big)}  [a_{j,-i+\hf}]
 \tOmega \big(t^{(A +E_{j,-i-\hf}^\theta-E_{j,-i +\hf}^\theta)}\big)
 \\
 =&\sum_{j\in [\hf-m,m-\hf]} q^{\sum_{k\in[\hf-m,j-1]}\big(a_{k,i-\hf}-  a_{k,i+\hf}\big)}  [a_{j,i-\hf}]
 \tOmega \big(t^{(A +E_{j,i+\hf}^\theta-E_{j,i-\hf}^\theta)}\big)
 \\
 =&t^{(\Omega A)}E_i,
\end{align*}
where the last equality follows from Propositions~\ref{prop:EFtA}.

It remains to show that $\tOmega$ is a $\Ui_m$-module homomorphism. 
By Propositions~\ref{prop:BtA}, we have for $i\in [1-m,m-1],i\neq 0$
\begin{align*}
\tOmega(B_i t^{(A)}) = &\sum_{j\in [\hf-n,n-\hf]} q^{\sum_{k\in[j,n-\hf]}
\big(a_{i+\hf,k}-  a_{i-\frac{1}{2},k}\big)+1}  [a_{i-\hf,j}]
 \tOmega \big(t^{(A +E_{i + \hf,j}^\theta-E_{i - \hf,j}^\theta)}\big)
 \\
 =&\sum_{j\in [\hf,n-\hf]} 
 q^{\sum_{k\in[j,n-\hf]}\big(a_{i+\hf,k}-  a_{i-\frac{1}{2},k}\big)+1}  [a_{i-\hf,j}]
 t^{(A +E_{i+ \hf,j}-E_{i - \hf,j})}
 \\
&+ \sum_{j\in [\hf,n-\hf]} q^{\sum_{k\in[ \hf-n,j]}\big(a_{-i-\hf,k}-  a_{-i+\frac{1}{2},k}\big)+1}  [a_{-i+\hf,j}]
 t^{(A +E_{-i- \hf,j}-E_{-i + \hf,j})}
 \\
 =&(F_i + E_{-i} K_i^{-1}) t^{(\Omega A)},
\end{align*}
where the last equality follows from Propositions~\ref{prop:EFtA} and the following formula
\begin{align*}
K_i^{-1} t^{(\Omega A)}=q^{\sum_{k\in[ \hf-n,-\hf]}\big(a_{-i-\hf,k}-  a_{-i+\hf,k}\big)}t^{(\Omega A)}.
\end{align*}

Using similar computation, one can also show that $\tOmega(B_0 t^{(A)})=B_0 t^{(\Omega A)}$. Therefore, the $\tOmega$ is a homomorphism of $(\Ui_m,\U(\sl_n))$-bimodules as desired.
\end{proof}

\begin{remark}\label{rmk:Umiso}
Using similar arguments, one can construct an isomorphism $\tOmega':\cV_{m|n,d}^\imath\cong \cV_{m|N,d}$ of $\big(\U(\sl_m),\Ui_n\big)$-bimodules, as well as the following commutative diagram
\begin{center}
\begin{tikzcd}
  & \U(\sl_m) \arrow[hookrightarrow]{d}{\iota_m}
  & \curvearrowright\qquad  \mathcal{V}_{m|N,d} \qquad \curvearrowleft
  &  \U(\gl_N) \\
  & \Ui_m 
  & \curvearrowright\qquad  \mathcal{V}_{m|n,d}^\imath \arrow{u}{\tOmega'}\qquad \curvearrowleft
  &\Ui_n \arrow[hookrightarrow]{u}
\end{tikzcd}
\end{center}
\end{remark}

\begin{corollary}\label{cor:int}
The $(\Ui_m,\Ui_n)$-bimodule $\cV_{m|n,d}^{\imath}$ is an integrable left $\Ui_m$-module and an integrable right $\Ui_n$-module.
\end{corollary}

\begin{proof}
Recall from \cite{WZ25} that any $\Ui_m$-module obtained by taking the restriction of a finite-dimensional $\U(\gl_M)$-module is an integrable $\Ui_m$-module. In particular, $\cV_{M|n,d}$ is an integrable $\Ui_m$-module.
By Theorem~\ref{thm:Uniso}, $\tOmega:\cV_{m|n,d}^\imath\cong\cV_{M|n,d}$ is an isomorphism of $\Ui_m$-modules and hence the $\Ui_m$-module $\cV_{m|n,d}^\imath$ is integrable. Using the isomorphism $\tOmega'$ in Remark~\ref{rmk:Umiso} and similar arguments as above, one can also show that the $\Ui_n$-module $\cV_{m|n,d}^{\imath}$ is integrable.
\end{proof}

By Proposition~\ref{prop:braidmod} and Corollary~\ref{cor:int}, the space $\cV_{m|n,n}^{\imath,0}=\big(\cV_{m|n,n}^{\imath}\big)_{\ov{\rho}}$ admits actions of $\TT_i$ for $i\in [0,n-1]$. Let $\uTT_i$ be the image of $\TT_i$ in $\End(\mathcal{V}_{m|n,n}^{\imath,0})$.

\begin{proposition}\label{prop:TiHecke}
Let $i\in[1,n-1]$. Under the isomorphism $\iPhi$, the action of $-q\TT_i$ on $\cV_{m|n,n}^{\imath,0}$ coincides with the action of $\sigma_i$ on $V^{\otimes n}$. In particular, we have
\begin{align}\label{eq:TiHecke}
(\uTT_i+q^{-2})(\uTT_i-1)=0.
\end{align}
\end{proposition}

\begin{proof}
The space $\cV_{M|n,n}^0$ of $\sl_n$-weight $0$ vectors in $\cV_{M|n,n}$ is defined as (cf. \cite{Zha02})
\[
\cV_{M|n,n}^0=\{v\in \cV_{M|n,n}|v K_i =v, \forall i\in[1,n-1] \}. 
\]
Comparing this definition with \eqref{def:0iwt}, one can identify$\mathcal{V}_{m|n,n}^{\imath,0}$ with $\cV_{M|n,n}^0$ under the isomorphism $\tOmega$. Moreover, under the embedding $\iota_n$, the relative braid group symmetry $\TT_i,i> 0$ is identified with Lusztig's braid group symmetry $T_i=T_{i,-1}'$; cf. \cite[Section 4]{WZ25}. Using formulas in \cite[(4.7)]{Zha02}, the action of $(-q)T_i$ on the space $\cV_{M|n,n}^0$ is identified with the action of $\sigma_i$ on $V^{\otimes n}$ (under the isomorphism $\tOmega \circ \iPhi:V^{\otimes n}\cong\cV_{M|n,n}^0$). By Theorem~\ref{thm:Uniso}, the action of $-q\TT_i$ on $\cV_{m|n,n}^{\imath,0}$ coincides with the action of $\sigma_i$ on $V^{\otimes n}$.
\end{proof}

\subsection{Relative braid group symmetry $\TT_0$}
In this section, we show that the relative braid group symmetry $\TT_0$ on $\mathcal{V}_{m|n,n}^{\imath,0}$ satisfies the Hecke relation. 
Recall from Proposition~\ref{prop:braidmod}(1) that the action of $\TT_0$ on $\mathcal{V}_{m|n,n}^{\imath,0}$ is given by the formula \eqref{eq:Tmodi}.
 %Let $\uTT_0$ be the image of $\TT_0$ in $\End(\mathcal{V}_{m|n,n}^{\imath,0})$.

 Recall the isomorphism $\iPhi:V^{\otimes n} \cong\cV_{m|n,n}^{\imath,0}$ from the proof of Theorem~\ref{thm:0wti}.

\begin{theorem}\label{thm:T0Hecke}
Under the isomorphism $\iPhi$, the action of $-q\TT_0$ on $\cV_{m|n,n}^{\imath,0}$ coincides with the action of $\sigma_0$ on $V^{\otimes n}$. In particular, we have
\begin{align}\label{eq:T0Hecke}
(\uTT_0+q^{-2})(\uTT_0-1)=0.
\end{align}
\end{theorem}

%We show that the action of $\TT_0$ on $\mathcal{V}_{m|n,n}^{\imath,0}$ satisfies the quadratic Hecke relation.

\begin{proof}
It suffices to assume that $n=1$ since the formula \eqref{eq:B0tA2} only involves columns $\pm\frac{1}{2}$. Since there is an unique nonzero entry on each column of $A$ for $A\in \Xi_{m|1,1}^{\imath,0}$, $A$ must take the form $E_{j,\frac{1}{2}}^\theta$ for some $j\in \I_{2m}$. We shall simply write $v_j$ for $\iPhi(v_j)= t^{(E_{j,\frac{1}{2}}^\theta)}$ below. Then we have
\begin{align*}
&\sum_{k\in[j+1,m-\frac{1}{2}]}\big(a_{k,\frac{1}{2}}-  a_{k,-\frac{1}{2}}\big)=-\delta_{j<0},
\\
&\sum_{j\in[\frac{1}{2},m-\frac{1}{2}]} \big( a_{j,\frac{1}{2}}-a_{j,-\frac{1}{2}}\big)=\delta_{j>0}-\delta_{j<0}.
\end{align*}
Hence, the identity \eqref{eq:B0tA2} is simplified as
\begin{align}
\begin{split}
v_j B_0 = v_{-j}+ q v_j,\qquad
v_{-j} B_0 = v_j + q^{-1} v_{-j}.
\end{split}
\end{align}
for $j>0$. In particular, the subspace $\cV_j$ spanned by $\{ v_j,v_{-j}\}$ for each fixed $j\in[\hf,m-\hf]$ is invariant under the action of $B_0$ and $B_0$ acts on this subspace by
\begin{align}\label{eq:B0matrix}
\begin{pmatrix}
q & 1
\\
1 & q^{-1}
\end{pmatrix}
\sim
\begin{pmatrix}
[2] &0
\\
0 & 0
\end{pmatrix}
\end{align}
where the second matrix is obtained by changing bases. 

Since $\cV_{m|1,1}^{\imath,0}=\bigoplus_{j\in[\hf,m-\hf]} \cV_j$, any eigenvalue of $B_0$ on $\mathcal{V}_{m|1,1}^{\imath,0}$ is either $0$ or $[2]$. By \eqref{def:idv} and \eqref{eq:B0matrix}, we have $\edvi{2p}\cV_j=0$ for $p>1$. It follows from \eqref{eq:Tmodi} that, when acting on $\mathcal{V}_j$,
\begin{align}\label{eq:T0matrix}
\uTT_0 =1 -q^{-1} \edvi{2} =
\begin{pmatrix}
0 & -q^{-1}
\\
-q^{-1} & 1-q^{-2}
\end{pmatrix}
\sim
\begin{pmatrix}
-q^{-2} &0
\\
0 & 1
\end{pmatrix}.
\end{align}
i.e.,
\[
(-q\uTT_0) v_j =v_{-j},\qquad (-q\uTT_0) v_{-j} =v_{-j}+(q^{-1}-q)v_{-j}.
\]
Comparing these formulas with \eqref{eq:sigma}, it is clear that the action of $-q\TT_0$ on $\cV_{m|n,n}^{\imath,0}$ coincides with the action of $\sigma_0$ on $V^{\otimes n}$. The relation \eqref{eq:T0Hecke} follows since $(\sigma_0+q)(\sigma-q^{-1})=0$.
%In other word, t Since $j$ is arbitrary and $\cV_{m|1,1}^{\imath,0}=\bigoplus_{j\in[\hf,m-\hf]} \cV_j$, we conclude that $\uTT_0$ satisfies the desired relation \eqref{eq:T0Hecke}.
\end{proof}

As a consequence of Proposition~\ref{prop:TiHecke} and Theorem~\ref{thm:T0Hecke}, we have the following result.
\begin{corollary}\label{cor:Endi}
The operators $\uTT_i,i\in [0,n-1]$ generates $\End_{\Ui_m}\big(\cV_{m|n,n}^{\imath,0}\big)$.
\end{corollary}

\begin{proof}
Clearly,  $\uTT_i\in\End_{\Ui_m}\big(\cV_{m|n,n}^{\imath,0}\big)$. By Propositions~\ref{prop:braidmod}, \ref{prop:TiHecke} and Theorem~\ref{thm:T0Hecke}, there is an algebra homomorphism $\cH_B(n)\rightarrow \End_{\Ui_m}\big(\cV_{m|n,n}^{\imath,0}\big)$ given by $\sigma_i\mapsto \uTT_i,i\in [0,n-1]$. Moreover, one can directly check that this map is injective. On the other hand, by Proposition~\ref{prop:ischur} and Theorem~\ref{thm:0wti}, $\End_{\Ui_m}\big(\cV_{m|n,n}^{\imath,0}\big)\cong \cH_B(n)$. By comparing dimensions, it is clear that $\End_{\Ui_m}\big(\cV_{m|n,n}^{\imath,0}\big)$ is generated by $\uTT_i,i\in [0,n-1]$.
\end{proof}

\begin{remark}[An action of the Hecke algebra with unequal parameters]
Let $s\in \Z$. Following \cite{Let02}, one can also formulate the nonstandard iquantum group $\Ui_{n,s}$ with the parameter $s$ ; $\Ui_{n,s}$ is defined to be the subalgebra of $\U(\gl_{2n})$ generated by 
\[
B_{0,s}=E_0+q F_0 K_0^{-1}+[s]K_0^{-1}
\]
and $B_i,d_r$ for $i\in \I_{2n-1},i\neq 0,r\in [\hf,n-\hf]$. Note that $\Ui_{n}$ defined in Section~\ref{sec:iQG} corresponds to $\Ui_{n,1}$. In fact, $\Ui_{n,s}$ is isomorphic to $\Ui_n$ as algebras for any $s$; see \cite[Section 9]{Ko14}. 

By similar calculations as in the proof of Theorem~\ref{thm:T0Hecke}, the action of $B_{0,s}$ on the subspace $\{v_j,v_{-j}\}$ is given by
\begin{align*}
\begin{pmatrix}
q [s] &1
\\
1 & q^{-1} [s]
\end{pmatrix}
\sim
\begin{pmatrix}
[s+1] &0
\\
0 & [s-1]
\end{pmatrix}.
\end{align*}
%When $s=0$, $\uTT_0$ satisfies $\uTT_0^2=1$. When $s=3$, $\uTT_0$ satisfies $(\uTT_0+q^{-2})(\uTT_0-q^{-8})=0$. 
Using \eqref{eq:Tmodi}, for an arbitrary odd integer $s=2r-1$, we have
\begin{align}
(\uTT_0+(-1)^{r}q^{-2r^2})(\uTT_0+(-1)^{r-1}q^{-2(r-1)^2})=0;
\end{align}
for an arbitrary even integer $s=2r$, we have
\begin{align}
(\uTT_0+(-1)^{r}q^{-2r(r+1)})(\uTT_0+(-1)^{r-1}q^{-2r(r-1) })=0.
\end{align}
By taking a renormalization of $\uTT_0$, the above two identities imply that
\begin{align}\label{eq:unequal}
(\uTT_0+q^{-2s})(\uTT_0-1)=0.
\end{align}
The formula \eqref{eq:unequal} together with Proposition~\ref{prop:TiHecke} shows that $\TT_i$ induces an action of the type B Hecke algebra $\cH_B(n;q^s,q)$ with unequal parameters (cf. \cite{Lus03}) on the space $\cV_{m|n,n}^{\imath,0}\cong V^{\otimes n}$. This $\cH_B(n;q^s,q)$-action clearly commutes with the left $\Ui_m$-action, and actually we have recovered the following multiparameter iSchur duality 
\begin{align}\label{eq:multpar}
\Ui_m \curvearrowright V^{\otimes n} \curvearrowleft \cH_B(n;q^s,q).
\end{align}
This multiparameter iSchur duality was first established in \cite{BWW18}.

In particular, when $s=0$, \eqref{eq:multpar} recovers the iSchur duality established in \cite[Section 3.1]{Bao17}. Moreover, since the type D Hecke algebra $\cH_D(n)$ can be realized as a subalgebra of $\cH_B(n;1,q)$ (see \cite[Section 3.2]{Bao17}), one can follow the arguments {\em loc. cit.} to recover the following iSchur duality of type D
\begin{align} 
\Ui_m \curvearrowright V^{\otimes n} \curvearrowleft \cH_D(n).
\end{align}
\end{remark}

\subsection{Idempotents in the quantum Schur algebra}\label{sec:Schuralg}

%\begin{theorem}
%The $(\Ui_m,\Ui_n)$-duality in Proposition~\ref{prop:BtA} implies the $(\Ui_m,\cH_B(n))$-duality in \eqref{eq:SchurBi}.
%\end{theorem}

%We consider the $\Ui_n$-weight space $\cV_{m|n,n}^{\imath,0}$ in \eqref{def:0iwt}. By Theorem~\ref{thm:0wti} and its proof, there is an isomorphism $\iPhi:V^{\otimes n} \cong\cV_{m|n,n}^{\imath,0}$ as left $\Ui_m$-modules. By Theorem~\ref{thm:T0Hecke} and Proposition~\ref{prop:TiHecke}, under the isomorphism $\iPhi$, the actions of $-q\TT_i$ for $i\in [0,n-1]$ on $\cV_{m|n,n}^{\imath,0}$ are identified with the actions of $\sigma_i$ in \eqref{eq:sigma}. Hence, $\cV_{m|n,n}^{\imath,0}$ is equipped with an action of $\cH_B(n)$ and this $\cH_B(n)$-action commutes with the $\Ui_m$-action. 

%It remains to show the double centralizer property. 
Recall that $M=2m$ is set to be even in this section.
Let $\Psi:\Ui_n\rightarrow \End(\cV_{m|n,n}^{\imath})$ denote the action of $\Ui_n$ on $\cV_{m|n,n}^{\imath}$. Let $S^B_{M,n}=\End_{\cH_B(n)}(V^{\otimes n})$ be the type B quantum Schur algebra (recall from Section~\ref{sec:ischur} that $\dim V=M$). By \cite[Theorem 4.3]{LX22}, the action of $\Ui_n$ on $\cV_{m|n,n}^{\imath}$ factors through $S^B_{M,n}$ and 
\begin{align}\label{eq:schur1}
\End_{\Ui_m}(\cV_{m|n,n}^{\imath})=\Psi(S^B_{M,n}).
\end{align}

Let $e:\cV_{m|n,n}^{\imath}\rightarrow \cV_{m|n,n}^{\imath,0}\hookrightarrow \cV_{m|n,n}^{\imath}$ be the projection map. Clearly, $e$ is $\Ui_m$-equivariant. By \eqref{eq:schur1}, we can view $e$ as an idempotent in $\Psi(S^B_{M,n})$ and we have
\[
\End_{\Ui_m}(\cV_{m|n,n}^{\imath,0}) = e\Psi(S^B_{M,n})e.
\]
Let $e^B$ be the idempotent in $S^B_{M,n}$ constructed in \cite[Proposition 4.1.1]{LNX22}. By \eqref{def:0iwt} and the definition of $e^B$ {\em loc. cit.}, it is straightforward to check that the action of $e^B$ on $\cV_{m|n,n}^{\imath}$ is exactly given by the projection $e$. i.e. $\Psi(e^B)=e$.

It was shown in \cite[Proposition 4.1.1]{LNX22} that $e^B S^B_{M,n} e^B\cong \cH_B(n)$. On the other hand, by Proposition~\ref{prop:ischur} and Theorem~\ref{thm:0wti}, $\End_{\Ui_m}\big(\cV_{m|n,n}^{\imath,0}\big)\cong \cH_B(n)$. Therefore, $\Psi$ sends $e^B S^B_{M,n} e^B$ isomorphically to $\End_{\Ui_m}(\cV_{m|n,n}^{\imath,0})$.

\section{$\mathrm{i}$Howe dual pair $(\Uj_m,\Uj_n)$ and $\mathrm{i}$Schur duality}\label{sec:jHowe}

In this section, we fix $m,n\in \N$ and set $N=2n+1,M=2m+1$.  We show that the iSchur duality \eqref{eq:SchurBj} between $\Uj_m$ and $\cH_B(n)$ can be recovered from the iHowe dual pair $(\Uj_m,\Uj_n)$ in \cite{LX22}. 

\subsection{$(\Uj_m,\Uj_n)$-iHowe duality}\label{sec:Ujduality}
We set up the $(\Uj_m,\Uj_n)$-iHowe duality following \cite{LX22}.  
Set
\begin{align*}
&\Xi_{m|n,d}^\jmath=\{(a_{ij})_{i\in[-m,m],j\in[-n,n]}\in \Mat_{M\times N}(\N)|a_{ij}=a_{-i,-j},\sum_{i,j} a_{ij}=2d+1\}.
\\
&\Xi_{m|n}^\jmath=\bigsqcup_{d\in \N}\Xi_{m|n,d}^\jmath.
\end{align*}

Let $\mathcal{J}$ be the right ideal of $\cT_{2n+1}$ (see \eqref{def:cT} for its definition) generated by
\begin{align*}
&t_{i,j}-t_{-i,-j}-(q-q^{-1}) t_{i,-j},\quad t_{i,0}-q t_{-i,0},
\\
& t_{0,j}-q t_{0,-j},\quad t_{i,-j}-t_{-i,j},\quad (0<i,j\leq n),
\end{align*}
and define $\mathcal{T}_{n}^\jmath:=\mathcal{T}_{2n+1}/\mathcal{J}$. $\mathcal{T}_{n}^\jmath$ has a coalgebra structure but not an algebra structure. The space $\cV_{m|n}^\jmath$ is defined to be the subspace of $\mathcal{T}_{\max(m,n)}^\jmath$ spanned by $\{t^{(A)}|A\in\Xi_{m|n}^\jmath \}$ and $\cV_{m|n,d}^\jmath$ is defined to be the subspace of $\mathcal{T}_{\max(m,n)}^\jmath$ spanned by $\{t^{(A)}|A\in\Xi_{m|n,d}^\jmath \}$, where $t^{(A)}=t_{0,0}^{(a_{0,0}-1)/2} \prod_{(i,j)>(0,0)}^< t_{ij}^{a_{ij}}$. By \cite[Theorem 5.3]{LX22}, $\{t^{(A)}|A\in \Xi_{m|n,d}^\jmath\}$ is a basis for $\cV_{m|n,d}^\jmath$ and $\cV_{m|n}^\jmath=\bigoplus_{d\in \N} \cV_{m|n,d}^\jmath$.

\begin{proposition}[\text{cf. \cite[Proposition 5.5, Theorem 4.3]{LX22}}]
\label{prop:jEFTA}
Let $d\in \N$. There is a (left) $\Uj_m$ action on $\mathcal{V}_{m|n,d}^\jmath$ given as follows: for $i\in [\hf,m-\hf]$, $A\in \Xi_{m|n,d}^\jmath$,
\begin{align}\notag
B_i t^{(A)}=&\sum_{\substack{j\in [-n,0]\\a_{i-\hf,j}>\delta_{i,\hf}\delta_{j0}}} q^{\sum_{k\in[j,n]} (a_{i+\hf ,k}-a_{i-\hf,k})+1+\delta_{i,\hf}}
[a_{i-\hf,j} -\delta_{i,\hf}\delta_{j,0}]t^{\Big(A+E_{i+\hf,j}^\theta-E_{i-\hf,j}^\theta\Big)},
\\\label{eq:BitA1j}
&+\sum_{ j\in [1,n] } q^{\sum_{k\in[j,n]} (a_{i+\hf ,k}-a_{i-\hf,k})+1}
[a_{i-\hf,j}  ]t^{\Big(A+E_{i+\hf,j}^\theta-E_{i-\hf,j}^\theta\Big)},
\\ 
B_{-i} t^{(A)}=&\sum_{j\in [-n,n]} q^{\sum_{k\in[-n,j]} (a_{i-\hf ,k}-a_{i+\hf,k})+1}
[a_{i+\hf,j}] t^{\Big(A+E_{i-\hf,j}^\theta-E_{i+\hf,j}^\theta\Big)}.
\label{eq:BitA2j}
\end{align}
There is a right $\Uj_n$ action on $\mathcal{V}_{m|n,d}^\jmath$ given as follows: for $i\in [\hf,n-\hf]$, $A\in \Xi_{m|n,d}^\jmath$,
\begin{align}
t^{(A)} B_i &=\sum_{j\in [-m,m]} 
q^{\sum_{k \in [j+1,m]} (a_{k, i+\hf }-a_{k, i-\hf})}\big[a_{j,i+\hf} \big]t^{\Big(A+E_{ j ,i-\hf}^\theta-E_{j,i+\hf}^\theta\Big)},
\\\notag
t^{(A)} B_{-i}&=\sum _{\substack{j\in [-m,0]\\a_{j,i-\hf}> \delta_{i,\hf}\delta_{j,0}}} q^{\sum_{k \in [-m,j-1]} (a_{k,i-\hf}-a_{k,i+\hf})}
[a_{j,i-\hf}-\delta_{i,\hf}\delta_{j,0} ]t^{\Big(A+E_{j,i+\hf}^\theta-E_{j,i-\hf}^\theta\Big)}
\\
&\quad +\sum _{j\in [1,m]} q^{\sum_{k \in [-m,j-1]} (a_{k,i-\hf}-a_{k,i+\hf})-\delta_{i,\hf}}
[a_{j,i-\hf}]t^{\Big(A+E_{j,i+\hf}^\theta-E_{j,i-\hf}^\theta\Big)}.
\end{align}
The above two actions
\begin{align}\label{eq:Ujduality}
\Uj_m \curvearrowright \mathcal{V}_{m|n,d}^\jmath \curvearrowleft \Uj_n
\end{align}
satisfy the double centralizer property.
\end{proposition}

The $\Uj_n$-weight space $\big(\cV_{m|n,d}^{\jmath}\big)_{\ov{\mu}}$ for $\mu=\sum_{r\in \I_N}\mu_r\epsilon_r$ is given by
\[
\big(\cV_{m|n,d}^{\imath}\big)_{\ov{\mu}}=\{v\in \cV_{m|n,d}^{\imath}|vd_0=q^{\mu_0}v, v d_r=q^{\mu_r+\mu_{-r}}v,\forall r\in [1,n]\}.
\]
Then $\cV_{m|n,d}^{\jmath}=\bigoplus_{\ov{\mu}\in X_\imath} \big(\cV_{m|n,d}^{\jmath}\big)_{\ov{\mu}}$. Note that the value of $\mu_0,\mu_r+\mu_{-r}$ does not depend on the choice of a preimage of $\ov{\mu}$ in $X$.

\begin{proposition}\label{prop:intj}
The space $\cV_{m|n,d}^\jmath$ is integrable viewed as a left $\Uj_m$-module and also integrable viewed as a right $\Uj_n$-module.
\end{proposition}

\begin{proof}
We show that $\mathcal{V}_{m|n,d}^\jmath$ is integrable viewed as a left $\Uj_m$-module; the integrability for its right $\Uj_n$-module structure can be proved in the same way.

Let $i\in [\hf,m-\hf]$ and $ A\in \Xi_{m|n,d}$. By \eqref{eq:BitA1j}, $B_i^k t^{(A)}=0$ if $k>\sum_{j\in [-n,n]}a_{i+\hf,j}$. By \eqref{eq:BitA2j}, $B_{-i}^k t^{(A)}=0$ if $k>\sum_{j\in [-n,n]}a_{i-\hf,j}$. The integrability is proved.
\end{proof}

By Proposition~\ref{prop:braidmod} and Proposition~\ref{prop:intj}, the space $\cV_{m|n,d}^\jmath$ is equipped with the action of relative braid group symmetries $\TT_i$ for $i\in [0,n-1]$. 

\subsection{The iweight $\ov{\rho}$ space}

We define $\cV_{m|n,n}^{\jmath,0}$ to be the $\Uj_n$-weight space $\big(\cV_{m|n,n}^{\jmath}\big)_{\ov{\rho}}$; see \eqref{def:rho} for the definition of $\rho$. Explicitly, 
\begin{align}\label{def:0jwt}
\cV_{m|n,n}^{\jmath,0}=\{v\in\cV_{m|n,n}^{\jmath}|vk_{i}=q^{-\delta_{i,\hf}}v, \forall i\in [\hf,n-\hf]\}.
\end{align}
%Since $k_i=K_{\epsilon_i-\epsilon_{i+1}+\epsilon_{-i}-\epsilon_{-i-1}}$, 
By definition, the action of $d_r,k_i$ for $r\in [1,n], i\in [\hf,n-\hf]$ on $\cV_{m|n,d}^\jmath$ is given by
\begin{align}\label{eq:ktA2}
\begin{split}
& t^{(A)}d_r =q^{\sum_{k\in \I_{2m+1}}  a_{k,r}  }t^{(A)},
\qquad  t^{(A)} d_0=q^{(-1+\sum_{k\in \I_{2m+1}}  a_{k,0})/2 } t^{(A)},
\\
& t^{(A)}k_i=q^{\sum_{k\in \I_{2m+1}} (a_{k,i-\hf}-a_{k,i+\hf})-\delta_{i,\hf}}t^{(A)},\qquad \forall A\in\Xi_{m|n,d}^{\jmath}.
\end{split}
\end{align}

Due to \eqref{eq:ktA2}, $t^{(A)}$ lies in $\mathcal{V}_{m|n,n}^{\jmath,0}$ if and only if $\sum_{i} a_{i,j}=1$ for any $j\in[0,n]$. This condition further implies that $a_{0,j}=a_{0,-j}=0$ for $j\neq 0$ and $a_{0,0}=1$. Hence, the space $\cV_{m|n,n}^{\jmath,0}$ admits a basis $\{t^{(A)}|A\in \Xi_{m|n,n}^{\jmath,0}\}$ where
\begin{align}
\Xi_{m|n,n}^{\jmath,0}=\big\{(a_{ij})\in \Xi_{m|n,n}^\jmath| a_{0,0}=1,\sum_{i} a_{i,j}=1,\forall j\in [1,n]\big\}.
\end{align} 

Let $V=\bC^M$ be the natural representation of $\U_M$, viewed as a $\Uj_m$-module via restriction.

\begin{theorem}\label{thm:0wtj}
There is an isomorphism of (left) $\Uj_m$-modules $\jPhi:\cV_{m|n,n}^{\jmath,0}\cong V^{\otimes n}$.
\end{theorem}

\begin{proof}

Let $\{ v_i|i\in \I_{2m+1} \}$ be the natural basis of $V$; see \eqref{eq:natural}. For $\lambda=(\lambda_1,\cdots,\lambda_n)\in \I_{2m+1}^n$, set
$v_\lambda=v_{\lambda_1} \otimes v_{\lambda_2} \otimes \cdots \otimes v_{\lambda_n}.$ The set $\{v_\lambda|\lambda\in \I_{2m+1}^n\}$ forms a basis of $V^{\otimes n}$.

Denote by $e_i=(0,\cdots,0,1,0\cdots,0)^t$ the $M\times 1$ column vector whose $i$-th entry is $1$. Then we have the following bijection between bases of $ V^{\otimes n}$ and $\cV_{m|n,n}^{\jmath,0}$ 
\begin{align*}
\I_{2m+1}^n &\rightarrow \Xi_{m|n,n}^{\jmath,0}
\\
\lambda & \mapsto A_\lambda:=(e_{-\lambda_{n}},\cdots, e_{-\lambda_1},e_0,e_{\lambda_1},\cdots, e_{\lambda_n}).
\end{align*}
Hence, we have a linear isomorphism $\jPhi:V^{\otimes n}\cong \cV_{m|n,n}^{\jmath,0} , v_\lambda \mapsto t^{(A_\lambda)}$.

We show that $\jPhi$ is an isomorphism of $\Uj_m$-modules.  Write $A_\lambda=(a_{ij}^\lambda)$. Note that $D_r v_{\lambda_j}
  =q^{a_{r,j}^\lambda} v_{\lambda_j}$ for $r\in \I_{2m+1}$ and hence $K_i v_{\lambda_j}
  =q^{a_{i-\hf,j}^\lambda-a_{i+\hf,j}^\lambda} v_{\lambda_j}$ for $i\in \I_{2m},j\in [1,n]$. Using these formulas, it is direct to show that $\jPhi(d_r v_\lambda)=d_r\jPhi( v_\lambda)$ for $r\in [0,n]$.
  
Recall that $\Delta(B_i)=B_i\otimes K_i^{-1} + 1\otimes F_i + k_i^{-1}\otimes  E_{-i}K_i^{-1}  $ and $\Delta(k_i)=k_i\otimes k_i$ for $i\in [\hf,m-\hf]$. Let $\Delta^{{k}}$ be the algebra homomorphism $\U\rightarrow \underbrace{\U\otimes \cdots \otimes\U}_{k+1}$ induced from the coproduct $\Delta$. By induction, we have for $i\in [\hf,m-\hf]$
 \begin{align}\notag
 \Delta^{{n-1}}(B_i)=&\sum_{r=1}^{n} \underbrace{1\otimes \cdots \otimes 1}_{r-1} \otimes F_i \otimes \underbrace{K_i^{-1}\otimes \cdots \otimes K_i^{-1}}_{n-r}
 \\
 &+\sum_{r=1}^{n} \underbrace{k_i^{-1} \otimes \cdots \otimes k_i^{-1}}_{r-1} 
 \otimes  E_{-i}K_i^{-1} \otimes\underbrace{ K_i^{-1} \otimes \cdots \otimes K_i^{-1}}_{n-r}.
 \label{eq:Delnj}
 \end{align}

 Using $\eqref{eq:Delnj}$, we have for $i\in [\hf,m-\hf]$
 \begin{align}
 B_i v_{\lambda}=&\sum_{r=1}^n q^{\sum_{k>r}\big(a_{i+\hf,k}^\lambda -a_{i-\hf,k}^\lambda\big)}  \delta_{\lambda_r=i-\hf} v_{\lambda+\epsilon^r}
 \\\notag
 &+ q^{\sum_{k=1}^n \big(a_{i+\hf,k}^\lambda -a_{i-\hf,k}^\lambda\big)}  \sum_{r=1}^n
 q^{\sum_{1\leq k<r}\big(a_{-i-\hf,k}^\lambda - a_{-i+\hf,k }^\lambda\big)} \delta_{\lambda_r=-i+\hf} v_{\lambda-\epsilon^r},
 \end{align}
 where $\epsilon^r=(0,\cdots,0,1,0,\cdots ,0)\in \N^n$.

 For $A_\lambda\in  \Xi_{m|n,n}^{\jmath,0}$, we have $\delta_{\lambda_r=i-\hf}=\big[a^\lambda_{i-\hf,r}\big]$ and $\delta_{\lambda_r=-i+\hf}=\big[a^\lambda_{i-\hf,-r}\big]$. Note also that $\big[a^\lambda_{i-\hf,0}-\delta_{i,\hf}\big]=0$ for $i\in [\hf,n-\hf]$. Then we can rewrite the above formula as follows
  \begin{align}\notag
 B_i v_{\lambda}=&\sum_{r=1}^n q^{\sum_{k>r}\big(a_{i+\hf,k}^\lambda -a_{i-\hf,k}^\lambda\big)}  \big[a^\lambda_{i-\hf,r}\big] v_{\lambda+\epsilon^r}
 \\\notag
 &+ q^{\sum_{k=1}^n \big(a_{i+\hf,k}^\lambda -a_{i-\hf,k}^\lambda\big)}  \sum_{r=-n}^{-1}
 q^{\sum_{r<k\le -1}\big(a_{i+\hf,k}^\lambda - a_{i-\hf,k }^\lambda\big)}\big[a^\lambda_{i-\hf,r}\big] v_{\lambda-\epsilon^{-r}}
 \\\notag
 =&\sum_{r=1}^n q^{\sum_{k>r}\big(a_{i+\hf,k}^\lambda -a_{i-\hf,k}^\lambda\big)}  \big[a^\lambda_{i-\hf,r}\big] v_{\lambda+\epsilon^r}
 \\
&+ \sum_{r=-n}^{0}
 q^{\sum_{k>r}\big(a_{i+\hf,k}^\lambda - a_{i-\hf,k }^\lambda\big)+\delta_{i,\hf}}
 \big[a^\lambda_{i-\hf,r}-\delta_{i,\hf}\delta_{r,0}\big] v_{\lambda-\epsilon^{-r}}.
 \label{eq:Bvlaj}
 \end{align}

 By definition, $\jPhi(v_{\lambda\pm \epsilon^j})=t^{(A_\lambda+E_{\lambda_j\pm 1,j}^\theta-E_{\lambda_j,j}^\theta)}$. Note that if $a_{i-\hf,r}^\lambda\neq 0$, then $a_{i+\hf,r}^\lambda-a_{i-\hf,r}^\lambda+1=0$. Hence, \eqref{eq:Bvlaj} implies that
 \begin{align}\notag
 \jPhi(B_i v_{\lambda})
 =&\sum_{r=1}^n q^{\sum_{k\ge r}\big(a_{i+\hf,k}^\lambda -a_{i-\hf,k}^\lambda\big)+1}  
 \big[a^\lambda_{i-\hf,r}\big] 
 t^{\big(A_\lambda+E_{i+\hf,r}^\theta-E_{i-\hf,r}^\theta\big)}
 \\\notag
&+ \sum_{r=-n}^{0}
 q^{\sum_{k\ge r}\big(a_{i+\hf,k}^\lambda - a_{i-\hf,k }^\lambda\big)+1+\delta_{i,\hf}}
 \big[a^\lambda_{i-\hf,r}-\delta_{i,\hf}\delta_{r,0}\big] 
 t^{\big(A_\lambda+E_{-i-\hf,-r}^\theta-E_{-i+\hf,-r}^\theta\big)}
 \\\notag
= &\sum_{r=1}^n q^{\sum_{k\ge r}\big(a_{i+\hf,k}^\lambda -a_{i-\hf,k}^\lambda\big)+1}  \big[a^\lambda_{i-\hf,r}\big] t^{\big(A_\lambda+E_{i+\hf,r}^\theta-E_{i-\hf,r}^\theta\big)}
\\
&+ \sum_{r=-n}^{0}
 q^{\sum_{k\ge r}\big(a_{i+\hf,k}^\lambda - a_{i-\hf,k }^\lambda\big)+1+\delta_{i,\hf}}
 \big[a^\lambda_{i-\hf,r}-\delta_{i,\hf}\delta_{r,0}\big] 
 t^{\big(A_\lambda+E_{i+\hf,r}^\theta-E_{i-\hf,r}^\theta\big)}.
\label{eq:PhiBvj}
 \end{align}
 where the last equality follows by $E_{k,l}^\theta=E_{-k,-l}^\theta$. Now, one can compare \eqref{eq:PhiBvj} with \eqref{eq:BitA1j}, and then it is clear that $ \jPhi(B_i v_{\lambda})= B_i\jPhi( v_{\lambda})$ for $i\in [\hf,m-\hf]$.

 It remains to show that $ \jPhi(B_{-i} v_{\lambda})= B_{-i}\jPhi( v_{\lambda})$ for $i\in [\hf,m-\hf]$. Similar to \eqref{eq:Delnj}, we have
 \begin{align}\notag
 \Delta^{{n-1}}(B_{-i})=&\sum_{r=1}^{n} \underbrace{1\otimes \cdots \otimes 1}_{r-1} \otimes F_{-i} \otimes \underbrace{K_{-i}^{-1}\otimes \cdots \otimes K_{-i}^{-1}}_{n-r}
 \\
 &+\sum_{r=1}^{n} \underbrace{k_{i} \otimes \cdots \otimes k_i}_{r-1} \otimes K_{-i}^{-1} E_{i}  \otimes\underbrace{ K_{-i}^{-1} \otimes \cdots \otimes K_{-i}^{-1}}_{n-r}.
 \label{eq:Deln2j}
 \end{align}
 Using \eqref{eq:Deln2j}, we have
 \begin{align}\notag
 B_{-i}v_\lambda=&\sum_{r=1}^n q^{\sum_{k>r}\big(a_{-i+\hf,k}^\lambda -a_{-i-\hf,k}^\lambda\big)}  \delta_{\lambda_r=-i-\hf} v_{\lambda+\epsilon^r}
 \\\notag
 &+ q^{\delta_{i,\hf}+\sum_{k=1}^n \big(a_{-i+\hf,k}^\lambda -a_{-i-\hf,k}^\lambda\big)}  \sum_{r=1}^n
 q^{\sum_{1\leq k<r}\big(a_{i-\hf,k}^\lambda - a_{i+\hf,k }^\lambda\big)} 
 \delta_{\lambda_r=i+\hf} v_{\lambda-\epsilon^r}
 \\\notag
 =&\sum_{r\in[-n,-1]} q^{\sum_{k<r}\big(a_{i-\hf,k}^\lambda -a_{i+\hf,k}^\lambda\big)}  \delta_{\lambda_{-r}=-i-\hf} v_{\lambda+\epsilon^{-r}}
 \\
 &+ q^{\delta_{i,\hf}+\sum_{k=-n}^{-1} \big(a_{i-\hf,k}^\lambda -a_{i+\hf,k}^\lambda\big)} 
 \sum_{r=1}^n q^{\sum_{1\leq k<r}\big(a_{i-\hf,k}^\lambda - a_{i+\hf,k }^\lambda\big)} 
 \delta_{\lambda_r=i+\hf} v_{\lambda-\epsilon^r}.
 \end{align}
 For $i\in [\hf,m-\hf]$, we have $a_{i-\hf,0}^\lambda - a_{i+\hf,0}^\lambda=\delta_{i,\hf}$. Note also that if $a_{i+\hf,r}^\lambda\neq 0$, then $a_{i-\hf,r}^\lambda -a_{i+\hf,r}^\lambda+1=0$. Hence, similar to the first case, we can rewrite the above formula as
 \begin{align}\notag
 B_{-i}v_\lambda=&\sum_{r\in[-n,-1]} q^{\sum_{k\le r}\big(a_{i-\hf,k}^\lambda -a_{i+\hf,k}^\lambda\big)+1} \big[a_{i+\hf,r}^\lambda\big] v_{\lambda-\epsilon^{-r}}
 \\\label{eq:PhiBvj2}
 &+    \sum_{r\in [1,n]}
 q^{\sum_{k\le r}\big(a_{i-\hf,k}^\lambda - a_{i+\hf,k }^\lambda\big)+1} 
 \big[a_{i+\hf,r}^\lambda\big] v_{\lambda+\epsilon^r}.
 \end{align}
 Moreover, for $A\in \Xi_{m|n,n}^{\jmath,0}$, the $j=0$ component in the right-hand side of \eqref{eq:BitA2j} is $0$. Thus, by comparing \eqref{eq:BitA2j} and \eqref{eq:PhiBvj2}, it is clear that $\jPhi(B_{-i} v_{\lambda})= B_{-i}\jPhi( v_{\lambda})$ for $i\in [\hf,m-\hf]$.

 Therefore, we have proved that $\jPhi$ is a $\Uj_m$-module isomorphism.
\end{proof}

\subsection{Relative braid group symmetry $\TT_0$}
 By Proposition~\ref{prop:braidmod}(3), the space $\cV_{m|n,n}^{\jmath,0}$ is invariant under the actions of $\TT_i$ for $i\in [0,n-1]$. In this section, we show that the relative braid group symmetry $\TT_0$ on $\cV_{m|n,n}^{\jmath,0}$ satisfies the Hecke relation. Let $\uTT_i$ be the image of $\TT_i$ in $\End(\cV_{m|n,n}^{\jmath,0})$ for $i\in [0,n-1]$.
 
 Recall the isomorphism $\jPhi:V^{\otimes n} \cong\cV_{m|n,n}^{\jmath,0}$ from the proof of Theorem~\ref{thm:0wtj}.
 
\begin{theorem}\label{thm:T0Heckej}
Under the isomorphism $\jPhi$, the action of $-q\TT_0$ on $\cV_{m|n,n}^{\jmath,0}$ coincides with the action of $\sigma_0$ on $V^{\otimes n}$. In particular, we have
\begin{align}
\label{eq:T0Heckej}
(\uTT_0+q^{-2})(\uTT_0-1)=0.
\end{align}
\end{theorem}

\begin{proof}
It suffices to assume that $n=1$ since the formula \eqref{eq:TT0} only involves $B_{\pm \hf}$.
Recall that there is a unique nonzero entry on each column of $A$ for $A\in \Xi_{m|n,n}^{\jmath,0}$. Hence, when $n=1$, the matrix $A$ must take the form $E_{j,1}^\theta+E_{0,0}$ for some $j\in [-m,m]$. We write $v_j$ for $\jPhi(v_j)=t^{(E_{j,1}^\theta+E_{0,0})}$.
We have
\begin{align*}
&\sum_{k < j} (a_{k,0}-a_{k,1})=\delta_{j>0},
\qquad
\sum_{k > j} (a_{k,1 }-a_{k, 0})=-\delta_{j< 0}.
\end{align*}

Assume that $j\neq 0$. It follows by Proposition~\ref{prop:jEFTA} that
\begin{align*}
v_j B_{\hf} B_{-\hf}=q v_j + v_{-j},\qquad v_{-j} B_{\hf} B_{-\hf}= v_j +q^{-1} v_{-j}.
\end{align*}
for any fixed $j>0$. Then $\{v_j,v_{-j}\}$ spans a subspace $\mathcal{V}_j$ which is invariant under the action of $B_{\hf}B_{-\hf}$.
Moreover, by Proposition~\ref{prop:jEFTA}, we have $v_s B_{\hf}^{(l)}=0$ and $v_s B_{-\hf} =0$ for $l>1,s\in [-m,m]$. 
Hence, by the formula \eqref{eq:TT0}, the element $ v_j \uTT_0$ is a linear combination of $v_j$ and $v_j B_{\hf}B_{-\hf}$. In particular, the space $\cV_j$ is invariant under the $\uTT_0$-action.

Using the formula \eqref{eq:TT0}, we compute the action of $\uTT_0$ on $\cV_j$ as follows
\begin{align}\label{eq:T0matrixj}
\uTT_0 = 1-q^{-1}
\begin{pmatrix}
q & 1 \\ 1 & q^{-1}
\end{pmatrix}
=
\begin{pmatrix}
0 & -q^{-1} \\ -q^{-1} & 1-q^{-2}
\end{pmatrix}.
%\sim
%\begin{pmatrix}
%-q^{-2} & 0 \\0 & 1
%\end{pmatrix}.
\end{align}

Assume that $j=0$. By Proposition~\ref{prop:jEFTA}, we have
\begin{align*}
v_0 B_{\hf}B_{-\hf} =  [2] v_0.
\end{align*}
Then the action of $\uTT_0$ on $v_0$ is given by 
\begin{align}\label{eq:T0matrixj2}
v_0 \uTT_0 = (1-q^{-1}[2])v_0=-q^{-2} v_0.
\end{align}
Note that $\cV_{m|1,1}^{\jmath,0}=\bC v_0 \oplus\big(\bigoplus_{j\in[1,m ]} \cV_j\big)$. By \eqref{eq:T0matrixj}-\eqref{eq:T0matrixj2} and \eqref{eq:sigma}, the action of $-q\TT_0$ on $\cV_{m|1,1}^{\jmath,0}$ coincides with the action of $\sigma_0$. The identity \eqref{eq:T0Heckej} follows since $(\sigma_0+q)(\sigma_0-q^{-1})=0$. 
\end{proof}

\subsection{Relative braid group symmetry $\TT_i,i>0$}
In this section, we show that relative braid group symmetries $\TT_i,i>0$ on $\mathcal{V}_{m|n,n}^{\jmath,0}$ satisfy the Hecke relation.

%In this subsection, we use $[\frac{3}{2},n-\hf]$ to label the simple roots of $\U(\sl_n)$; note that this is different from the convention in Section~\ref{sec:QG}.
%We label the simple roots of $\U(\gl_{M})$ in the same way as Section~\ref{sec:QG}. The reason of using this labeling for $\U(\sl_n)$ can be understood as follows.

%Observe from \eqref{diag:ev} that the type AIII$_{2n}$ Satake diagram contains a diagonal-type subdiagram $A_{n-1}\times A_{n-1}$ (removing the vertices $\pm \hf$). Equivalently, the subalgebra of $\Uj_n$ generated by $B_i,B_{-i},K_i$ for $i\in[\frac{3}{2},n-\hf]$ is isomorphic to $\U(\sl_n)$; cf. \cite[Theorem 7.4]{Let02}. Thus, with this labeling, we have a natural algebra embedding
%\begin{align}\label{def:jn}
%\jmath_n:\U(\sl_n)\hookrightarrow \Uj_n,\quad  E_i \mapsto B_i,\quad F_i \mapsto B_{-i}, \quad K_i \mapsto k_i,\quad (i\in[\frac{3}{2},n-\hf]).
%\end{align}

\begin{proposition}\label{prop:TiHeckej}
Let $i\in [1,n-1]$. Under the isomorphism $\jPhi$, the action of $-q\TT_i$ on $\cV_{m|n,n}^{\jmath,0}$ coincides with the action of $\sigma_i$ on $V^{\otimes n}$. In particular, we have
\begin{align}
\label{eq:TiHeckej}
(\uTT_i+q^{-2})(\uTT_i-1)=0.
\end{align}
\end{proposition}

\begin{proof}
Let $\lambda=(\lambda_1,\ldots,\lambda_n)\in \I_{2m+1}^n$ and write $t^{(A)}$ for $\jPhi(v_\lambda)$. Proposition~\ref{prop:jEFTA} implies that $t^{(A)} B_{i+\hf}^2=0$ for $A\in \Xi_{m|n,n}^{\jmath,0}$, and hence $t^{(A)}\TT_i=t^{(A)}(1-q^{-1} B_{i+\hf}B_{-i-\hf})$. By Proposition~\ref{prop:jEFTA}, we have
\begin{align*}
t^{(A)} B_{i+\hf}&=q^{-\delta_{\lambda_{i+1}< \lambda_i}} t^{(A+E_{\la_{i+1},i}^\theta-E_{\la_{i+1},i+1}^\theta)},
\\
t^{(A+E_{\la_{i+1},i}^\theta-E_{\la_{i+1},i+1}^\theta)} B_{-i-\hf}&=
\begin{cases}
t^{(A+E_{\la_{i},i+1}^\theta-E_{\la_i,i}^\theta+E_{\la_{i+1},i}^\theta-E_{\la_{i+1},i+1}^\theta)} +qt^{(A)},& \text{ if } \la_{i+1}>\la_i,
\\
[2] t^{(A)}, &\text{ if } \la_{i+1}=\la_i,
 \\
 q t^{(A+E_{\la_{i},i+1}^\theta-E_{\la_i,i}^\theta+E_{\la_{i+1},i}^\theta-E_{\la_{i+1},i+1}^\theta)} + t^{(A)}, &\text{ if } \la_{i+1}<\la_i.
\end{cases}
\end{align*}
By the definition of $\jPhi$, $t^{(A+E_{\la_{i},i+1}^\theta-E_{\la_i,i}^\theta+E_{\la_{i+1},i}^\theta-E_{\la_{i+1},i+1}^\theta)} $ is exactly $\jPhi(v_\la\cdot \sigma_i)$.  Hence, we have
\[
\jPhi(v_\lambda)\TT_i=
\begin{cases}
-q^{-1}\jPhi(v_\la\cdot \sigma_i), & \text{ if } \la_{i+1}>\la_i,
\\
-q^{-2}\jPhi(v_\la\cdot \sigma_i), & \text{ if } \la_{i+1}=\la_i,
\\
-q^{-1}\jPhi(v_\la\cdot \sigma_i)+(1-q^{-2})\jPhi(v_\la ), & \text{ if } \la_{i+1}<\la_i.
\end{cases}
\]
On can compare the above formulas with \eqref{eq:sigma}, and then it is clear that the action of $-q\TT_i$ coincides with the action of $\sigma_i$ under $\jPhi$.
\end{proof}

As a consequence of Theorem~\ref{thm:T0Heckej} and Proposition~\ref{prop:TiHeckej}, we have the following result.
\begin{corollary}\label{cor:Endj}
The operators $\uTT_i,i\in [0,n-1]$ generates $\End_{\Uj_m}\big(\cV_{m|n,n}^{\jmath,0}\big)$.
\end{corollary}

\begin{proof}
The proof is parallel to the proof of Corollary~\ref{cor:Endi} and hence omitted.
\end{proof}

\subsection{Idempotents in the quantum Schur algebra}
Recall that $M=2m+1$ is odd. Set $e:\cV_{m|n,n}^{\jmath}\rightarrow \cV_{m|n,n}^{\jmath,0}$ to be the projection map. Let $S^B_{M,n}=\End_{\cH_B(n)}(V^{\otimes n})$ be the type B quantum Schur algebra and $e^B$ be the idempotent in $S^B_{M,n}$ constructed in \cite[Proposition 4.1.1]{LNX22}. Similar to the case that $M$ is even (see Section~\ref{sec:Schuralg}), one can also identify the projection $e$ with the action of $e^B$ on $\cV_{m|n,n}^{\jmath}$.

 %Let $\Psi:\Uj_n\rightarrow \End(\cV_{m|n,n}^{\jmath})$ denote the action of $\Uj_n$ on $\cV_{m|n,n}^{\jmath}$. By \cite[Theorem 4.3]{LX22}, the action of $\Ui_n$ on $\cV_{m|n,n}^{\jmath}$ factors through $S^B_{M,n}$ and 
%\begin{align}\label{eq:schur1j}
%\End_{\Uj_m}(\cV_{m|n,n}^{\jmath})=\Psi(S^B_{M,n}).
%\end{align}

% Clearly, $e$ is $\Uj_m$-equivariant. By \eqref{eq:schur1j}, we can view $e$ as an idempotent in $\Psi(S^B_{M,n})$ and we have
%\[\End_{\Uj_m}(\cV_{m|n,n}^{\jmath,0}) = e\Psi(S^B_{M,n})e.\]
%Let  By \eqref{def:0iwt} and the definition of $e^B$ {\em loc. cit.}, it is straightforward to check that the action of $e^B$ on $\cV_{m|n,n}^{\imath}$ is exactly given by the projection $e$.

\section{Multiplicity-free decompositions}\label{sec:decomp}

In this section, we recall from \cite{Wat21,LX22} the multiplicity-free decompositions in the iHowe duality and the iSchur duality. Using these decompositions, we show that the iweight $\ov{\rho}$ spaces of irreducible modules over iquantum groups are irreducible $\cH_B(n)$-modules.

\subsection{Multiplicity-free decomposition for iHowe duality}

In this subsection, we recall from \cite{Wat21} irreducible modules for $\Uj_n$ and $\Ui_n$, and then recall from \cite{LX22} the multiplicity-free decompositions in the $(\Uj_m,\Uj_n)$-iHowe duality and $(\Ui_m,\Ui_n)$-iHowe duality. The notations will mainly follow \cite{LX22}. 

Let us first consider the case of $\Uj_n$. Set 
\[
t_0= [B_{\hf},B_{-\hf}]_q-\frac{k_{\hf}-k_{\hf}^{-1}}{q-q^{-1}}, 
\qquad 
t_i= \TT_i \cdots \TT_1(t_0),
\qquad
(1\le i\le n-1).
\]

Let $\A=\Q[q,q^{-1}]_{(q-1)}$ be the localization of $\Q[q,q^{-1}]$ at $(q-1)$. For any $(\ba,\bb)=(a_0,\ldots,a_n,b_1,\ldots,b_{n})\in \Z^{n+1}\times \A^n$, there is an unique irreducible left ({\em resp.} right) $\Uj_n$-module $L_{\ba,\bb}^{[n],\jmath}$ ({\em resp.} $\widetilde{L}_{\ba,\bb}^{[n],\jmath}$) such that $L_{\ba,\bb}^{[n],\jmath}$ is generated by a (highest weight) vector $v$ and 
\begin{align*}
&d_0 v =q^{a_0} v, \quad d_i v =q^{a_i} v,\quad t_{i-1} v = b_{i} v,\quad B_j v=0
\\
(resp.\qquad  &v d_0 =q^{a_0} v, \quad v d_i  =q^{a_i} v,\quad  v t_{i-1} = b_i v,\quad v B_{-j}=0),
\end{align*}
for $i\in [1,n],j\in [\hf,n-\hf]$.

Let $\Par_n(d)$ be ths set of partitions of $d$ with $n$ parts and $\Par(d)=\bigsqcup_n \Par_n(d)$. Set 
\[
\Par_n^\jmath(d)=\bigsqcup_{l}\Par_{n+1}(d-l)\times \Par_{n}(l).
\]
We use $(\la_0^+,\la_1^+,\ldots,\la_n^+,\la^-_1,\ldots,\la^-_n)$ to denote components of an element $\la=(\lambda^+,\la^-)\in \jPar_n(d)$. For $\la=(\lambda^+,\la^-)\in \jPar_n(d)$, let $L_\la^{[n],\jmath}$ ({\em resp.} $\widetilde{L}_\la^{[n],\jmath}$)  be the irreducible left ({\em resp.} right) $\Uj_n$-module associated to 
\[
(q^{\la_0^+},q^{\la_1^+ +\la_1^-},\ldots,q^{\la_n^+ + \la_n^-},
[\la_1^+ -\la_1^-],\ldots,[\la_n^+ -\la_n^-]).
\]

%%%%%%%%%%%
Let us next consider the case of $\Ui_n$. Set 
\[ 
t_0=B_0,\qquad
t_i= \TT_i \cdots \TT_1(B_0),
\qquad
(1\le i\le n-1).
\]
For any $(\ba,\bb)=(a_1,\ldots,a_n,b_1,\ldots,b_{n})\in \Z^{n}\times \A^n$, there is an unique irreducible left ({\em resp.} right) $\Ui_n$-module $L_{\ba,\bb}^{[n],\imath}$ ({\em resp.} $\widetilde{L}_{\ba,\bb}^{[n],\imath}$) such that $L_{\ba,\bb}^{[n],\imath}$ is generated by a (highest weight) vector $v$ and 
\begin{align*}
& d_{i-\hf} v =q^{a_i} v,\quad t_{i-1} v = b_{i} v,\quad B_j v=0
\\
(resp.\qquad  & v d_{i-\hf}  =q^{a_i} v,\quad  v t_{i-1} = b_i v,\quad v B_{-j}=0),
\end{align*}
for $i\in [1,n],j\in [1,n-1]$. Set 
\[
\iPar_n(d)=\bigsqcup_{l}\Par_{n}(d-l)\times \Par_{n}(l).
\]
For $\la=(\lambda^+,\la^-)\in \iPar_n(d)$, let $L_\la^{[n],\imath}$ ({\em resp.} $\widetilde{L}_\la^{[n],\imath}$)  be the irreducible left ({\em resp.} right) $\Ui_n$-module associated to 
\[
(q^{\la_1^+ +\la_1^-},\ldots,q^{\la_n^+ + \la_n^-},
[d+\la_1^+ -\la_1^-],\ldots,[d+\la_n^+ -\la_n^-]).
\]

\begin{proposition}\cite[Theorem 6.4]{LX22}
As a $(\Uj_m,\Uj_n)$-bimodule, $\cV^{\jmath}_{m|n,d}$ admits the following multiplicity-free decomposition
\begin{align}\label{decomUj}
\cV^{\jmath}_{m|n,d}=\bigoplus_{\la\in \jPar_m(d)\cap \jPar_n(d)} L_\la^{[m],\jmath} \otimes \widetilde{L}_\la^{[n],\jmath}.
\end{align}
As a $(\Ui_m,\Ui_n)$-bimodule, $\cV^{\imath}_{m|n,d}$ admits the following multiplicity-free decomposition
\begin{align}\label{decomUi}
\cV^{\imath}_{m|n,d}=\bigoplus_{\la\in \iPar_m(d)\cap \iPar_n(d)} L_\la^{[m],\imath} \otimes \widetilde{L}_\la^{[n],\imath}.
\end{align}
\end{proposition}

\subsection{Multiplicity-free decomposition for iSchur duality}\label{sec:decompSchur}
It is well-known that irreducible finite-dimensional modules for $\cH_A(n-1)$ are parameterized by $\Par(n)$. It was proved in \cite{DJ92} that irreducible finite-dimensional modules for the type B Hecke algebra $\cH_B(n)$ are parametrized by pairs of partitions $\bigsqcup_{k=0}^n \Par(k)\times \Par(n-k)$. In this section, we first recall from {\em loc. cit.} the construction of these irreducible $\cH_B(n)$-modules. Then using these irreducible $\cH_B(n)$-modules, we recall from \cite{Wat21} the multiplicity-free decomposition for iSchur duality.

For any $1\le i<j\le n-1$, set $\sigma_{i,j}=\sigma_i \sigma_{i+1}\cdots \sigma_{j-1}$ and $\sigma_{j,i}=\sigma_{j-1}\cdots\sigma_{i+1}\sigma_i$. Let $w_{k,n-k}=(\bs_{n-1}\cdots \bs_1)^{n-k}\in W^\circ$. Define 
\[
v_{k,n-k}=\prod_{i=1}^k (q+\sigma_{i,1}\sigma_0 \sigma_{1,i})\sigma_{w_{k,n-k}}
\prod_{j=1}^{n-k} (1-q\sigma_{j,1}\sigma_0 \sigma_{1,j}).
\]
We identify $\cH_A(k-1)\otimes\cH_A(n-k-1)$ as the subalgebra of $\cH_B(n)$ generated by $\sigma_i,i\in [1,n-1],i\neq k$. Following \cite[Definition 3.2.4]{DJ92}, there is an invertible element $\tz_{n-k,k}\in\cH_A(k-1)\otimes\cH_A(n-k-1)$ such that $v_{k,n-k}\sigma_{w_{n-k,k}} v_{k,n-k}=\tz_{k,n-k} v_{k,n-k}$. Set $e_{k,n-k}=v_{k,n-k} \sigma_{w_{n-k,k}} \tz_{n-k,k}^{-1}$ and $\varepsilon=\sum_{k=0}^n e_{k,n-k}$; one can show that $\varepsilon$ is an idempotent in $\cH_B(n)$. By \cite{DJ92}, there is an algebra isomorphism $\varepsilon\cH_B(n)\varepsilon\cong \bigoplus_{k=0}^n\cH_A(k-1)\otimes\cH_A(n-k-1)$, as well as an equivalence between the category of $\varepsilon\cH_B(n)\varepsilon$-modules and the category of $\cH_B(n)$-modules
\begin{align*}
\varepsilon\cH_B(n)\varepsilon \text{-mod} 
\overset{\sim}{\longrightarrow}\cH_B(n)\text{-mod},
\qquad
M\mapsto M\otimes_{\varepsilon\cH_B(n)\varepsilon }  \varepsilon\cH_B(n).
\end{align*}
Let $k\in [0,n]$ and $\la=(\la^+,\la^-)\in \Par(k)\times \Par(n-k)$. Denote by $S^{\la^+}$ ({\em resp.} $S^{\la^-}$) the irreducible $\cH_A(k-1)$-module ({\em resp.} $\cH_A(n-k-1)$-module) corresponding to $\la^+$ ({\em resp.} $\la^-$). Let $\cS^\la$ denote the irreducible $\cH_B(n)$-module corresponding to $S^{\la^+}\otimes S^{\la^-}$ under the above equivalence.

By Proposition~\ref{prop:ischur}, $V^{\otimes n}$ admits a multiplicity-free decomposition into irreducible $(\Uj_m,\cH_B(n))$-bimodules and the explicit description of simple bimodules appearing in this decomposition is given in \cite[Section 8.1 and Appendix A]{Wat21}.

\begin{proposition}\cite[Theorem A.3.2]{Wat21}
As a $(\Uj_m,\cH_B(n))$-bimodule, $V^{\otimes n}$ admits the following multiplicity-free decomposition
\begin{align}\label{eq:decompSchur}
V^{\otimes n}=\bigoplus_{\la\in \jPar_m(n)} L_\la^{[m],\jmath} \otimes \cS^{\la}.
\end{align} 
\end{proposition}

\subsection{The iweight $\ov{\rho}$ space as a $\cH_B(n)$-module}

Let us consider the $(\Uj_m,\Uj_n)$-duality. Note that $\jPar_m(n)\cap \jPar_n(n)= \jPar_m(n)$. Taking the $\Uj_n$-weight space for the $\Uj_n$-weight $\ov{\rho}$ in both sides of \eqref{decomUj}, we have
\begin{align}\label{eq:0wtdecomp}
\cV^{\jmath,0}_{m|n,n}=\bigoplus_{\la\in \jPar_m(n) } L_\la^{[m],\jmath} \otimes \big(\widetilde{L}_\la^{[n],\jmath}\big)_{\ov{\rho}}.
\end{align}

By Theorem~\ref{thm:T0Heckej} and Proposition~\ref{prop:TiHeckej}, $\cV^{\jmath,0}_{m|n,n}$ admits an action of $\cH_B(n)$ via the relative braid group symmetries $\TT_i,i\in [0,n]$. Since the action of $\TT_i$ are defined by elements in $\Ui_n$, each direct summand on the right-hand side is preserved by this $\cH_B(n)$-action, and hence we have an action of $\cH_B(n)$ on $\big(\widetilde{L}_\la^{[n],\jmath}\big)_{\ov{\rho}}$.

Similarly, using Theorem~\ref{thm:T0Hecke}, Proposition~\ref{prop:TiHecke} and  \eqref{decomUi}, one can equip $\big(\widetilde{L}_\la^{[n],\imath}\big)_{\ov{\rho}}$ with an $\cH_B(n)$-action.

\begin{theorem}\label{thm:mfdecomp}
The $\Uj_n$-weight space $(\widetilde{L}_\la^{[n],\jmath})_{\ov{\rho}}$ for any $\lambda\in \jPar_n(n)$ is an irreducible $\cH_B(n)$-module. The $\Ui_n$-weight space $(\widetilde{L}_\la^{[n],\imath})_{\ov{\rho}}$ for any $\lambda\in \iPar_n(n)$ is an irreducible $\cH_B(n)$-module. Moreover, we have the following isomorphism of  $\cH_B(n)$-modules
\begin{align}
(\widetilde{L}_\la^{[n],\jmath})_{\ov{\rho}}\cong\cS^{\la}.
\end{align}
\end{theorem}

\begin{proof}
Recall the isomorphism $\jPhi:V^{\otimes n} \rightarrow \cV_{m|n,n}^{\jmath,0}$ of $\Uj_m$-modules from the proof of Theorem~\ref{thm:0wtj}. By Theorem~\ref{thm:T0Heckej} and Proposition~\ref{prop:TiHeckej}, $\jPhi$ is an isomorphism of $(\Uj_m,\cH_B(n))$-bimodules.

Now compare \eqref{eq:0wtdecomp} and \eqref{eq:decompSchur}, and by the uniqueness, it is clear that $(\widetilde{L}_\la^{[n],\jmath})_{\ov{\rho}}$ is isomorphic to $\cS^{\la}$ as $\cH_B(n)$-modules.

Similarly, by Theorem~\ref{thm:T0Hecke} and Proposition~\ref{prop:TiHecke}, $\iPhi:V^{\otimes n} \rightarrow \cV_{m|n,n}^{\imath,0}$ defined in the proof of Theorem~\ref{thm:0wti} is an isomorphism of $(\Ui_m,\cH_B(n))$-bimodules. Note that $\iPar_m(n)\cap \iPar_n(n)= \iPar_m(n)$. By \eqref{decomUi}, we have
\begin{align} 
\cV^{\imath,0}_{m|n,n}=\bigoplus_{\la\in \iPar_m(n) } L_\la^{[m],\imath} \otimes \big(\widetilde{L}_\la^{[n],\imath}\big)_{\ov{\rho}}.
\end{align}
On the other hand, by Proposition~\ref{prop:ischur}, $V^{\otimes n}$ admits a multiplicity-free decomposition into irreducible $(\Uj_m,\cH_B(n))$-bimodules. Hence, using the isomorphism $\iPhi$, it is clear that $\big(\widetilde{L}_\la^{[n],\imath}\big)_{\ov{\rho}}$ are irreducible $\cH_B(n)$-modules.
\end{proof}

\section{Braid group actions and $K$-matrix}\label{sec:K}

In this section, we consider the $(\Ui_m,\Ui_n)$-iHowe duality. We construct in Section~\ref{sec:reflec} a type B braid group action on $\cV_{m|n,d}^\imath$ using $K$-matrices and $R$-matrices. We show in Theorem~\ref{thm:braidK} that, up to scalars, this action can be identified with the relative braid group action induced by $\TT_i,i\in [0,n-1]$.

\subsection{Quasi $K$-matrix and universal $K$-matrix}

Let $M=2m$. We recall the quasi $K$-matrix and universal $K$-matrix associated to the quantum symmetric pair $(\U_M=\U(\gl_M),\Ui_m)$ from the literature. 

Let $Q_M=\bigoplus_{i\in \I_{M-1}}\Z\alpha_i$ be the root lattice of $\gl_M$ and $Q^+_M=\bigoplus_{i\in \I_{M-1}}\N\alpha_i$. The lattice $Q_M$ is naturally a sublattice of the weight lattice $X_M$ of $\gl_M$ with $\alpha_i=\varepsilon_{i-\hf}-\varepsilon_{i+\hf}$. Denote by $\langle\cdot, \cdot\rangle$ the nondegenerate bilinear pairing on $X_M$. Let $\U_M^+$ be the subalgebra of $\U_M$ generated by $E_i,i\in \I_{M-1}$ and $\U_M^+=\bigoplus_{\mu\in Q^+_M} \U_{M,\mu}^+$ be the weight space decomposition. Recall from Lemma~\ref{lem:inv} the anti-involution $\sigma$ on $\U_M$.
%\[
%\sigma:E_i\mapsto E_i, \quad F_i\mapsto F_i, \quad D_a\mapsto D_a^{-1},\quad (i\in \I_{M-1},a\in \I_M).
%\]

 \begin{proposition}[\text{\cite[Theorem 3.16]{WZ23}; cf. \cite{BW18a,BW18b,BK19,AV20}}]
   \label{prop:qK}
 There exists a unique element $\fX=\sum_{\mu \in Q^+_M} \fX^\mu$ for $\fX^\mu\in \U_{M,\mu}^+$, such that $\fX^0=1$ and the following identities hold:
 \begin{align}\label{eq:fX1av}
 B_{i}  \fX  &=  \fX   \sigma( B_i),\qquad %\\
 k_i \fX = \fX k_i,
 \end{align}
 for $i\in \I_{M-1}$.
 \end{proposition}

 The element $\fX$ is called the {\em quasi $K$-matrix} associated to $(\U_M,\Ui_m)$. 

 Let $U,V$ be finite-dimensional $\U_M$-modules. Let $R_{U,V}:U\otimes V\rightarrow V\otimes U$ be the $R$-matrix associated to $\U_M$. The map $R_{U,V}$ is an isomorphism of $\U_M$-modules for any $U,V$ (see \cite[Theorem 32.1.5]{Lus94}) and they satisfy the Yang-Baxter equation
 \[
(R_{V_2,V_3} \otimes 1)( 1\otimes R_{V_1,V_3}) (R_{V_1,V_2} \otimes 1)
=
(1\otimes R_{V_1,V_2} )( R_{V_1,V_3}\otimes 1) (1 \otimes R_{V_2,V_3}).
 \]
 
 Following \cite[Section 4.5]{BW18b}, we set $g:X_M \rightarrow \Q(q)$ to be a function on $X_M$ such that 
 \begin{align}\label{eq:g}
 g(\mu)=g(\mu-\alpha_i)q^{\langle \alpha_{\tau i}-\alpha_i,\mu\rangle+1-\delta_{i,\tau i}}.
 \end{align}
Such a function exists and the value of $g$ on a set of representatives of $X_M/Q_M$ can be made arbitrary.

 Let $T_i^M,i\in \I_{M-1}$ denote Lusztig braid group symmetry $T_{i,-1}'$ on left integrable $\U_M$-modules. For $w\in W(A_{M-1})$ with a reduced expression $w=s_{i_1}\cdots s_{i_k}$, we set $T_w^M=T^M_{i_1}\cdots T^M_{i_k}$. The symmetry $T_w^M$ is independent of the choice of reduced expression of $w$.
 
 Let $V$ be a finite-dimensional $\U_M$-module with a weight space decomposition $V=\bigoplus_{\mu\in X} V_\mu$. The function $g$ induces a linear map $\tg$ on $V$ such that $\tg(v)=g(\mu)v$ for any $v\in V_\mu$. Then (cf. \cite[Corollary 7.7]{BK19}\cite[Theorem 4.5]{BW18b})
 \begin{align}\label{eq:cK}
 \cK=\fX\cdot \tg\cdot T_{w_0}^M
 \end{align}
defines a linear operator $\cK_V$ on any finite-dimensional $\U_M$-module $V$. The operator $\cK_V$ is called the {\em $K$-matrix} on $V$ associated to $(\U_M,\Ui_m)$.

 \begin{proposition}\label{prop:Kmatrix}
 Let $U,V$ be finite-dimensional $\U_M$-modules.
 \begin{itemize}
 \item[(1)] \cite{BW18b,BK19} The operator $\cK_V:V\rightarrow V$ is an isomorphism of $\Ui_m$-modules.

 \item[(2)] \cite[Corollary 9.6]{BK19} The operator $\cK$ satisfy the following reflection equation when acting on $V\otimes U$
 \begin{align}
 (\cK_V\otimes 1) R_{U,V}  (\cK_U \otimes 1) R_{V,U}  = R_{U,V} (\cK_U\otimes 1) R_{V,U}  (\cK_V\otimes 1).
 \end{align}
 \end{itemize}
 \end{proposition}

\subsection{Braid group action arising from reflection equation}\label{sec:reflec}

By Theorem~\ref{thm:Uniso}, we have the following commutative diagram
\begin{equation}
\begin{tikzcd}
  & \U(\gl_M)
  & \curvearrowright\qquad  \cV_{M|n} \qquad \curvearrowleft
  &  \arrow[hookrightarrow]{d}{\iota_n}\U(\sl_n) \\
  & \Ui_m \arrow[hookrightarrow]{u}
  & \curvearrowright\qquad  \cV_{m|n}^\imath \arrow{u}{\tOmega}\qquad \curvearrowleft
  &\Ui_n
\end{tikzcd}
\end{equation}

Recall that $\cV_{M|n}$ is $\N$-graded via $\cV_{M|n}=\bigoplus_{d\in \N} \cV_{M|n,d}$. Since all $\cV_{M|n,d}$ are finite-dimensional, we have a well-defined linear operator $\cK_{\cV_{M|n}}$ on $\cV_{M|n}$.

Let $j\in[1,n]$. Following \cite{TL02}, there is an embedding of $\U(\gl_M)$-modules 
\[
c_j:\cV_{M|1}\rightarrow \cV_{M|n}, \qquad t_{i,\hf}\mapsto t_{i,j-\hf}.
\]

\begin{lemma}[\text{\cite[Theorem 5.4]{TL02}}]\label{lem:TL}
There is an isomorphism of $\N$-graded $\U(\gl_M)$-modules given by
\[
\cV_{M|1}^{\otimes n}\overset{\sim}{\longrightarrow}\cV_{M|n},\qquad v_1\otimes v_2\otimes \cdots \otimes v_n\mapsto c_1(v_1)c_2(v_2)\cdots c_{n}(v_n).
\]
\end{lemma}

Let $\cK_1$ denotes the $K$-matrix on the first tensor factor of $\cV_{M|1}^{\otimes n}$. i.e., 
\[
\cK_1= \cK_{\cV_{M|1}} \otimes 1\otimes \cdots \otimes 1.
\]
Let $R_{i,i+1}$ denote the $R$-matrix acting on the $i$-th and $i+1$-th tensor factors of $\cV_{M|1}^{\otimes n}$. By Proposition~\ref{prop:Kmatrix}, these operators $\cK_1,R_{i,i+1}$ define an action of the type B$_n$ braid group on $\cV_{M|n}$.

By definition \eqref{eq:cK}, it is clear that $\cK_1$ preserves the $\N$-grading on $\cV_{M|n}$ and hence we have an action of the type B$_n$ braid group on $\cV_{M|n}$ for each $d\in \N$.

\subsection{$\gl_2$-modules} In this subsection, we set $n=1$ and study the (right) action of $\Ui_1=\langle B_0\rangle \subset \U(\gl_2)$ on irreducible $\U(\gl_2)$-modules. We omit the index $i$ for $E_i,F_i,K_i$.

Let $L(d)$ denote the $d+1$ dimensional irreducible $\U(\gl_2)$-module and $\eta$ be a cyclic vector with $\eta E=0$. Set $v_k=[d-k]!\eta F^k$ for $0\le k\le d$. Then we have
\[
v_k E = [k] v_{k-1},\qquad v_k F=[d-k]v_{k+1},\qquad v_k K=q^{2k-d} v_k.
\]
Recall from Section~\ref{sec:iQG} that $B_0=F+q^{-1}E K^{-1} + K^{-1}$ and $\Ui_1=\langle B_0 \rangle$. Using the above formula, we have
\begin{align}\label{eq:gl2}
 v_k B_0= q^{d-2k}v_k+[d-k]v_{k+1}+ q^{d-2k+1} [k] v_{k-1}.
\end{align}

Following \cite[(7.2)]{WZ25}, the formula of the symmetry $\TT_0$ on $L(d)$ is given as follows: if $d$ is odd, then 
\[
v\TT_0 =v\sum_{p=0}^\infty (-q)^{-p} \edvi{2p};
\]
if $d$ is even, then 
\[
v\TT_0 =v\sum_{p=0}^\infty (-q)^{-p} \odvi{2p+1}.
\]

\begin{lemma}\cite[Theorem 3.9, (3.24) and (7.2)]{WZ25} \label{lem:gl2}
The symmetry $\TT_0$ sends $v_0$ to $(-q)^{\lfloor\frac{d+1}{2}\rfloor} v_d$.
\end{lemma}

\subsection{$K$-matrix and the symmetry $\TT_0$}
It was shown in \cite[Theorem 6.5]{TL02} that, up to scalars (depending on $d$), the $R$-matrix $R_{i,i+1}$ on $\cV_{M|n,d}$ coincides with Lusztig symmetries $T_i$ associated to $\U(\sl_n)$.

Recall from Proposition~\ref{prop:braidmod} the relative braid group symmetry $\TT_i,i\in [0,n-1]$ for $\Ui_n$. Under the embedding $\iota_n$ in \eqref{eq:iotan}, we can identify the action of $\TT_i$ on $\cV_{m|n,d}^\imath$ with the action of $T_i$ on $\cV_{M|n,d}$; see \cite[Section~4]{WZ25} or the proof of Proposition~\ref{prop:TiHecke}. Hence, up to a scalar, $R_{i,i+1}$ on $\cV_{M|n,d}$ can be identified with the action of $\TT_i$.

Note that $\{d\varepsilon_{m-\hf}+Q_M| d\in \N\}$ is a set of distinct element in $X_M/Q_M$.  Thus, there is a function $g:X_M\rightarrow \Q(q)$ which satisfy the condition \eqref{eq:g} and $g(d\varepsilon_{m-\hf})=1$ for all $d\in \N$. In the rest of this section, we always assume the function $g$ satisfy these two conditions.

\begin{theorem}\label{thm:braidK}
When acting on $\cV_{M|n,d}$, we have
\begin{align}
\cK_{1} =(-q)^{-d(M-1)-\lfloor\frac{d+1}{2}\rfloor} \TT_0.
\end{align}
\end{theorem}

\begin{proof}
We identify $\cV_{M|n}$ with $\cV_{M|1}^{\otimes n}$ via the isomorphism in Lemma~\ref{lem:TL}. By Proposition~\ref{prop:BtA}, $B_0\in \Ui_n$ acts on the first tensor factor of $\cV_{M|1}^{\otimes n}$, and so does $\TT_0$. Therefore, since both $\cK_1$ and $\TT_0$ only act on the first factor of $\cV_{M|1}^{\otimes n}$, it suffices to assume $n=1$.

By Proposition~\ref{prop:Kmatrix}, $\cK_1$ commutes with all elements in $\Ui_m$ when acting on $\cV_{M|1}$. On the other hand, it is clear that $\TT_0$ commutes with elements in $\Ui_m$ when acting on $\cV_{M|1}$. By \cite[Lemma~6.2]{BW18b}, $\cV_{M|1}$ is generated by lowest weight vectors as a $\Ui_m$-module (here weights mean $\gl_M$-weights). Hence, it suffices to show that $\cK_1$ and $\TT_0$ match on lowest weight vectors in $\cV_{M|1}$.

 Fix $d\in \N$. By Proposition~\ref{prop:EFtA}, up to scaling, the $\U(\gl_M)$-module $\cV_{M|1,d}$ has a unique lowest weight vector $t^{d E_{m-\hf,\hf} }$ and a unique highest weight vector $t^{d E_{\hf-m,\hf} }$. Write $v^{d_1,d_2}=t^{d_1 E_{m-\hf,\hf} + d_2 E_{\hf-m,\hf} }$ for $d_1+d_2=d$. By Proposition~\ref{prop:BtA}, we have
\begin{align}\label{eq:TT0K}
v^{d_1,d_2} \cdot B_0 = q^{d_1-d_2}v^{d_1,d_2} + [d_1] v^{d_1-1,d_2+1}+q^{d_1-d_2+1}[d_2]  v^{d_1+1,d_2-1}.
\end{align}
Hence, the space $V(d)$ spanned by $\{v^{d_1,d_2}|d_1+d_2=d\}$ is stable under the action of $B_0$.
Moreover, by \eqref{eq:gl2} and \eqref{eq:TT0K}, $V(d)$ is isomorphic to $L(d)$ as the right module over $\Ui_1=\langle B_0\rangle$ and the isomorphism is given by $v^{d-k,k}\mapsto v_k$. Then Lemma~\ref{lem:gl2} implies that $v^{d,0}\TT_0=(-q)^{\lfloor\frac{d+1}{2}\rfloor}v^{0,d}$.

On the other hand, by \eqref{eq:cK} and our choice of the function $g:X_M\rightarrow \Q(q)$, $\cK_1$ sends the lowest $\gl_M$-weight vector $v^{d,0}$ of $\cV_{M|1,d}$ to $(-q)^{-d(M-1)}v^{0,d}$. Therefore, we have $\cK_1 v^{d,0}= (-q)^{-d(M-1)-\lfloor\frac{d+1}{2}\rfloor} v^{0,d}\TT_0$ and this implies that $\cK_1=(-q)^{-d(M-1)-\lfloor\frac{d+1}{2}\rfloor} \TT_0$ on $\cV_{M|1,d}$ by the above arguments. 
%\cK_1 v^{d,0}=(-q)^{*} v^{0,d}.
\end{proof}

Combining Theorem~\ref{thm:braidK} with results in \cite{TL02}, we have the following result.

\begin{corollary}\label{cor:braidactions}
Up to scalars, the (type B) braid group action constructed in Section~\ref{sec:reflec} using $K$-matrices and $R$-matrices of $(\U_M,\Ui_m)$ coincides with the relative braid group action induced by $\TT_i,i\in [0,n-1]$.
\end{corollary}

\end{document}